\documentclass[a4paper]{article}

\usepackage{geometry}
\usepackage{amssymb, amsmath, latexsym, mathrsfs, verbatim, calc}
\usepackage{empheq}
\usepackage{color}
\usepackage{xcolor}
\usepackage{indentfirst}\label{key}
\usepackage{amsthm}
\usepackage[bookmarks=true, bookmarksnumbered=true , bookmarksopen=true, bookmarksopenlevel = 2, colorlinks, linkcolor=cyan, anchorcolor=cyan, citecolor=cyan]{hyperref}   	
\usepackage{color}
\usepackage{graphicx}
\usepackage{subfigure}
\usepackage[titletoc,title]{appendix}
\usepackage{cases}
\usepackage{float}

\newtheorem{definition}{Definition}[section]
\newtheorem{proposition}{Proposition}[section]
\newtheorem*{problem}{Problem}
\newtheorem{lemma}{Lemma}[section]
\newtheorem{example}{Example}[section]
\newtheorem{theorem}{Theorem}[section]
\newtheorem{corollary}{Corollary}[section]

\newenvironment{proofth3.1}{{\noindent\bf Proof of Theorem \ref{theorem.2}}\quad}{\hfill $\square$\par}
\newenvironment{proofth4.1}{{\noindent\bf Proof of Theorem \ref{theorem.3}}\quad}{\hfill $\square$\par}
\newenvironment{proofth4.2}{{\noindent\bf Proof of Theorem \ref{theorem.verification}}\quad}{\hfill $\square$\par}
\newenvironment{proofth4.3}{{\noindent\bf Proof of Theorem \ref{theorem.4}}\quad}{\hfill $\square$\par}

\newenvironment{mmvSN.proof.lemma.solution.suppose}{{\noindent\bf Proof of Proposition \ref{mmvSN.lemma.solution.suppose}}\quad}{\hfill $\square$\par}
\newenvironment{mmvSN.proof.lemma.relationship_of_XY}{{\noindent\bf Proof of Lemma \ref{mmvSN.lemma.relationship_of_XY}}\quad}{\hfill $\square$\par}
\newenvironment{mmvSN.proof.lemma.coefficients}{{\noindent\bf Proof of Lemma \ref{mmvSN.lemma.coefficients}}\quad}{\hfill $\square$\par}
\newenvironment{mmvSN.proof.theorem.solution.explicit}{{\noindent\bf Proof of Theorem \ref{mmvSN.theorem.solution.explicit}}\quad}{\hfill $\square$\par}
\newenvironment{mmfSN.proof.theorem.efficient_frontier}{{\noindent\bf Proof of Theorem \ref{mmfSN.theorem.efficient_frontier}}\quad}{\hfill $\square$\par}

\begin{document}
\title{\bf Optimal Monotone Mean-Variance Problem in a Catastrophe Insurance Model}


\author{
	\renewcommand{\thefootnote}{\arabic{footnote}}
	Bohan Li\footnotemark[1]
	\and
	\renewcommand{\thefootnote}{\arabic{footnote}}
	Junyi Guo\footnotemark[2]
	\and
	\renewcommand{\thefootnote}{\arabic{footnote}}
	Xiaoqing Liang\footnotemark[3]
}
\date{}

\renewcommand{\thefootnote}{\fnsymbol{footnote}}
\footnotetext{Bohan Li}
\footnotetext{bhli@suda.edu.cn}
\footnotetext{}
\footnotetext{Junyi Guo: corresponding author}
\footnotetext{jyguo@nankai.edu.cn}
\footnotetext{}
\footnotetext{Xiaoqing Liang}
\footnotetext{liangxiaoqing115@hotmail.com}
\footnotetext{}

\renewcommand{\thefootnote}{\arabic{footnote}}
\footnotetext[1]{Center for Financial Engineering, Soochow University, Suzhou, Jiangsu 215006, P.R.China}
\footnotetext[2]{School of Mathematical Sciences, Nankai University, Tianjin 300071, P.R.China}
\footnotetext[3]{School of Sciences, Hebei University of Technology, Tianjin 300401, P.R.China}

\maketitle

\begin{abstract}
\noindent This paper explores an optimal investment and reinsurance problem involving both ordinary and catastrophe insurance businesses. The catastrophic events are modeled as following a compound Poisson process, impacting the ordinary insurance business. The claim intensity for the ordinary insurance business is described using a Cox process with a shot-noise intensity, the jump of which is proportional to the size of the catastrophe event. This intensity increases when a catastrophe occurs and then decays over time. The insurer's objective is to maximize their terminal wealth under the Monotone Mean-Variance (MMV) criterion. In contrast to the classical Mean-Variance (MV) criterion, the MMV criterion is monotonic across its entire domain, aligning better with fundamental economic principles. We first formulate the original MMV optimization problem as an auxiliary zero-sum game. Through solving the Hamilton-Jacobi-Bellman-Isaacs (HJBI) equation, explicit forms of the value function and optimal strategies are obtained. Additionally, we provides the efficient frontier within the MMV criterion. Several numerical examples are presented to demonstrate the practical implications of the results.
\end{abstract}
\vfill





\noindent


\section{Introduction}

\par Large-scale disasters, especially natural catastrophes, have the potential to result in numerous injuries, fatalities, and substantial economic losses. For insurance companies, the occurrence of such catastrophic events significantly amplifies their exposure to claims within a short time. In recent times, some insurance firms have introduced specialized catastrophe insurance tailored to specific natural disasters. Therefore, the management of catastrophic risk has assumed an increasingly pivotal role within the insurance industry. In practice, insurance companies often establish a bankruptcy remote entity known as a Special Purpose Vehicle (SPV). They purchase reinsurance from the SPV through continuous premium payments. When a catastrophe occurs, the SPV partially compensates the insurance company for its losses. Simultaneously, the SPV packages and securitizes this catastrophic risk as catastrophe bonds, selling them to investors. These catastrophe bonds, whose risk depend on disaster events and often displays low correlation with financial market, become attractive risk diversification assets among investors.

\par Typically, optimal reinsurance problems model the insurer's surplus process using the Cramér-Lundberg risk model. However, catastrophe insurance claims differ significantly from those in traditional insurance models. Catastrophes lead to an immediate surge in claims, and even after the event, the repercussions persist and gradually recede over time. Furthermore, the frequency of ordinary claims rises in tandem with the severity of the catastrophe(see \cite{born2006catastrophic} for more details). To capture this phenomenon, we employ a Cox process, to model the occurrence process of ordinary claims whose intensity is a stochastic process following a shot-noise process, which represents the occurrence and receding of the catastrophe. While this model has been applied in asset pricing (see \cite{dassios2003pricing}, \cite{jang2003shot}, \cite{dassios2003pricing}, \cite{jaimungal2006catastrophe}, \cite{Egami2008Indifference}, \cite{Unger2010Pricing}, \cite{Zonggang2017Pricing} and \cite{Eichler2017Utility}), it has received limited attention in the context of catastrophe insurance optimization. \cite{Delong2007Mean} first considered the optimal investment and reinsurance problem for the shot-noise process. They used the diffusion approximation of the shot-noise process introduced by \cite{dassios2005kalman-bucy} to optimize the terminal wealth under the classical MV criterion. \cite{Bi2010Hedging} minimized the risk of unit-linked life insurance by hedging in a financial market driven by a discontinuous shot-noise process. \cite{brachetta2019optimal} investigated the optimal investment and reinsurance problem for a Cox process. In their model, the intensity of the claim arrival process is modelled by a diffusion process. The objective was to maximize the expected exponential utility of its terminal wealth. The value function and optimal strategy were provided via the solutions to backward partial differential equations. \cite{Cao2020Optimal} considered the self-exciting and externally-exciting contagion claim Model under the time-consistent MV criterion with a model first introduced by \cite{Dassios2011Adynamic}. Therefore, existing research primarily focuses on continuous-time models or diffusion approximation models. The optimization problem of catastrophe models with pure jump settings in both the claim process and the intensity process, especially under the emerging MMV criterion, remains an under-explored topic of study.

\par In this paper, we address this issue by considering an insurer managing both ordinary and catastrophe insurance claims under the MMV criterion within a Black-Scholes financial market framework. The insurer faces potential claims, which fall into two categories: ordinary claims and catastrophe claims. The modeling of ordinary aggregate claims involves the use of a Cox process, wherein the intensity process takes the form of a shot noise process. Simultaneously, catastrophe claims are modeled using a compound Poisson process, sharing the same Poisson counting process as the intensity of ordinary claims. This implies that catastrophe claims and the impact of disasters occur concurrently. We make an assumption that the sizes of catastrophe claims are proportional to the effect of the catastrophe, represented as the jump of the shot noise process. Specifically, ordinary claims are expressed as:
\begin{align*}
	C_1(t) = \sum_{i=1}^{N_1(t)}U_i
\end{align*}
where $N_1(t)$ is a Cox process with intensity process $\lambda(t)$. $U_i$ denotes the size of the $i$-th claim, with the claim sizes $\{U_i\}_{i\geq1}$ being independent and identical distributed. The intensity process $\lambda(t)$ is defined as 
\begin{align*}
	\lambda(t) := \lambda_0 e^{-\delta t} + \sum_{i=1}^{N_2(t)} V_i e^{-\delta (t-\tau_i)}.
\end{align*}
Here, $N_2(t)$ is a Poisson process with a constant intensity $\rho$, and $\{V_i\}_{i\geq1}$ are also identically and independently distributed random variables. The shot-noise process $\lambda(t)$ is used to describe the occurrence and receding pattern of the catastrophes. $V_i$ represents the impact of the $i$-th catastrophe, and it decays exponentially with the factor $\delta$. In practice, the intensity $\rho$ tends to be relatively low due to the infrequent occurrence of catastrophes. We introduce a constant scale factor $k$, thus $kV_i$ denotes the size of the $i$-th catastrophe claims, resulting in the aggregate catastrophe claims process:
\begin{align*}
	C_2(t) = \sum_{i=1}^{N_2(t)}kV_i.
\end{align*}
Consequently, the total risk assumed by the insurer is given by:
\begin{align*}
	C_1(t) + C_2(t).
\end{align*}

\par Despite the successful application of the MV criterion in finance and economics, it has a notable drawback of lacking monotonicity. In other words, situations may arise where $X < Y$ but $\mathbb{E}(X)-\frac{\theta}{2}\mathbb{V}ar(X) > \mathbb{E}(Y)-\frac{\theta}{2}\mathbb{V}ar(Y)$, where $\theta$ represents the investor's risk aversion parameter. To address this limitation, \cite{maccheroni2009portfolio} introduced the MMV criterion, which serves as the nearest monotonic counterpart to the MV criterion. Interestingly, the MMV criterion aligns with the classical MV criterion within its domain of monotonicity. While there are relatively few studies in the literature related to the MMV optimization criterion, some researchers have ventured into this area. For example, \cite{trybula2019continuous-time} explored this criterion in a continuous-time framework, assuming that coefficients of stock prices are random as functions of a stochastic process. They focused on the discounted terminal wealth process under the MMV criterion, obtaining optimal portfolio strategies and value functions under specified coefficient conditions. In our earlier research works, \cite{li2021optimal} and \cite{li2022Monotone}, we delved into the MMV optimization criterion within the context of investment and reinsurance problems. Specifically, we explored this criterion in two distinct scenarios: the diffusion approximation model and the Cramér-Lundberg risk model. It is noteworthy that we made an observation regarding the optimal strategies in these two frameworks that these optimal strategies align with those derived for the classical MV problem. In a significant contribution, \cite{strub2020note} and \cite{Ales2020Semimartingale}\footnote{\cite{Ales2020Semimartingale} shows that the
optimal wealth processes of MMV problems are identical to that of MV problems under a certain condition which is applicable for continuous assets processes (see \cite{delbaen1996variance}).} extended this understanding by demonstrating that for a broad class of portfolio choice problems, when risky assets are {\em continuous} semi-martingales, the optimal portfolio and value function under the classical MV criterion and the MMV criterion are equivalent, which is not the case in our present paper, since the catastrophe claim process is not continuous. To the best of our knowledge, no research has explored the optimization strategies under the MV or MMV criteria in the context of discontinuous catastrophe models. 

\par In light of these gaps, our paper focuses on an optimal reinsurance problem encompassing both ordinary and catastrophe insurance claims, operating under the MMV criterion. This model allows the insurer to invest within a Black-Scholes financial market. To mitigate claim-related risks, the insurer has the option to purchase reinsurance for both ordinary and catastrophe insurance from SPV. For analytical simplicity, we confine the reinsurance format to proportional reinsurance. The optimal MMV problem takes the form of a max-min problem. Initially, the objective is to minimize it by selecting an alternative probability measure that is absolutely continuous (though not necessarily equivalent) with respect to the reference probability measure $P$. Subsequently, the objective is maximized by determining the optimal insurer strategies. To facilitate this complex analysis, we adopt a procedure similar to \cite{li2021optimal} in which the alternative probability measure is replaced by its conditional expected Radon-Nikodym derivative with respect to the reference probability measure. Despite the inherent challenge of absolute continuity (not equivalence), this approach enables us to represent the conditional expected Radon-Nikodym derivative as the solution to a stochastic differential equation (see \cite{li2022Monotone}). This reformulation converts the original optimal MMV problem into a two-player non-zero sum game, incorporating the conditional expected Radon-Nikodym derivative into the state processes. We leverage dynamic programming techniques to address this auxiliary problem, deriving the explicit solution to the Hamilton-Jacobi-Bellman (HJB) equation after rigorous calculations. Additionally, we provide the efficient frontier for the MMV problem involving catastrophic insurance.

\par The structure of the paper is organized as follows: In Section 2, we formulate the insurance model and introduce the MMV criterion. Subsequently, we transform the MMV maximization problem into an auxiliary zero-sum game problem. Employing dynamic programming, we derive the HJBI equation satisfied by the value function. In Section 3, we further solve the HJBI equation, demonstrating that the candidate value function and strategies indeed represent the optimal value function and corresponding strategies. Toward the end of this section, we provide the efficient frontier under the MMV criterion. In Section 4, we investigate the diffusion approximation model within the MMV criterion framework. Section 5 offers numerical examples and sensitivity analyses to illustrate our findings. Finally, Section 6 draws conclusions from our research.

\section{Model Formulation}

\subsection{Insurance Model}

\par Let $P$ be the real-world probability measure and $R_0(t)$ be the surplus of an insurer, which operates both the ordinary insurance business and the catastrophe insurance business,
\begin{align*}
dR_0(t) = c(t) dt  -  dC_1(t) - dC_2(t),
\end{align*}
where $c(t)$ is the premium rate at time $t$. The aggregate claims of the insurer are composed of two parts, $C_1(t)$, denoting the ordinary aggregate claims, and $C_2(t)$, which represents the aggregate catastrophe claims. In the event of a catastrophe, there is an immediate surge in claims, which is characterized by the catastrophe claim process denoted as $C_2(t)$.

\par In traditional insurance literature, the aggregate claims process $C_1(t)$ is typically assumed to follow a compound Poisson process. However, considering the impact of catastrophes, the frequency of ordinary claims $C_1(t)$ should be contingent upon the occurrence of a catastrophe. The fixed intensity of the compound Poisson process is inadequate for capturing the dynamic nature of ordinary claim frequency. To address this limitation, we replace the constant intensity of $C_1(t)$ with a stochastic process denoted as $\lambda(t)$. Let $N_1(t)$ represent the counting process for $C_1(t)$; in this context, $N_1(t)$ can be described as a doubly stochastic Poisson process, also known as a Cox process\footnote{For more details, we refer the readers to the works of \cite{cox1955some}, \cite{bartlett1963the}, \cite{serfozo1972conditional}, \cite{bremaud1983point}, \cite{grandell2006doubly}, and \cite{grandell2012aspects}.}. .
\par Denote $\lambda(t)$ the shot noise process as follows\footnote{See \cite{cox1980point}, \cite{cox1986the}, \cite{kluppelberg1995explosive}, and \cite{dassios2003pricing}.}:
\begin{align*}
\lambda(t) = \lambda_0 e^{-\delta t} + \sum_{i=1}^{N_2(t)} V_i e^{-\delta (t-\tau_i)}.
\end{align*}
We employ $\lambda(t)$ to characterize the impact of catastrophes on the intensity of $N_1(t)$, where $\lambda_0$ corresponds to the initial value, $V_i$ stands for the impact of the $i$-th catastrophe, $N_2(t)$ represents the total number of catastrophes occurring before time $t$, $\tau_i$ denotes the time at which the $i$-th catastrophe transpires, and $\delta$ signifies the exponential decay factor for $\lambda$.

\par Let
\begin{align*}
C_1(t) = \sum_{i=1}^{N_1(t)}U_i, \qquad C_2(t) = \sum_{i=1}^{N_2(t)}kV_i.
\end{align*}
in which each $U_i$ represents the magnitude of the $i$-th ordinary claim. Additionally, the premium rate $c(t)$ is considered a random process, depends on the observations of $\lambda(s)$ within the interval $[0,t)$. This is because the insurer has the capability to monitor the historical data of past catastrophes and evaluate their lasting effects to determine the pricing of catastrophe insurance products. At this point, $c(t)$ is defined by
\begin{align*}
c(t) :=& (1+\kappa)\frac{d\mathbb{E}^P[C_1(t)|\lambda(s),0 \leq s \leq t]}{dt} + (1+\iota)\frac{d\mathbb{E}^P[C_2(t)]}{dt}  \\
=& (1+\kappa) \lambda(t) \mathbb{E}^P U_i + (1+\iota) \rho k \mathbb{E}^P V_i.
\end{align*}
where $\rho$ is the intensity of $N_2(t)$, $\kappa$ and $\iota$ denote the safety loadings of the insurer for ordinary claims and catastrophe claims, respectively. The integral forms of $\lambda(t)$ and $R_0(t)$ are\footnote{We use $\mathbb{R}_{>0}$ to denote the space of positive real numbers, and $\mathbb{R}_{\geq 0}$ to denote the space of non-negative real numbers.}
\begin{align*}
d\lambda(t) = -\delta \lambda(t)dt + \int_{\mathbb{R}_{>0}} z N_2(dt,dz), \quad \lambda(0) = \lambda_0,
\end{align*}
and
\begin{align*}
dR_0(t) = (1+\kappa) \lambda(t) \mathbb{E}^P U_i dt + (1+\iota) \rho k \mathbb{E}^P V_i dt -  \int_{\mathbb{R}_{>0}} z N_1(dt,dz) - \int_{\mathbb{R}_{>0}} k z N_2(dt,dz).
\end{align*}

\par The insurer has the option to acquire proportional reinsurance to mitigate both ordinary and catastrophic risks. We represent the safety loadings for reinsurance companies for ordinary and catastrophe claims as $\kappa_r$ and $\iota_r$, respectively. In this scenario, the managed surplus process evolves as:
\begin{align*}
dR(t) =&  [(1+\kappa_r)u(t) - (\kappa_r - \kappa)] \lambda(t) \mathbb{E}^P U_i dt -  \int_{\mathbb{R}_{>0}} z u(t) N_1(dt,dz) \\
&+ [(1+\iota_r)v(t) - (\iota_r - \iota)] \rho k \mathbb{E}^P V_i dt -  k\int_{\mathbb{R}_{>0}} z v(t) N_2(dt,dz),
\end{align*}
where the two control variables, $u(t)$ and $v(t)$, are the retention levels for ordinary claims and catastrophe claims, respectively.

\par We also allow the insurer to invest its wealth to a risky asset, the price of which at time $t$ is governed by a geometric Brownian motion given by
\begin{align*}
dS(t) =  \mu_0 S(t)dt +  \sigma_0 S(t)dW_0(t),
\end{align*}
where $\mu_0$ and $\sigma_0$ are constants, and $W_0(t)$ represents a standard Brownian motion. The insurer is also allowed to allocate its wealth into a risk-free asset, the price of which at time $t$ follows the process:
\begin{align*}
dB(t) =  r B(t)dt,
\end{align*}
where $r$ is a fixed interest rate. Let $\pi(t)$ denote the amount that the insurer invests into the risky asset, then the insurer's surplus process follows
\begin{align}
\label{mmvSN.welath.0}
dX(t) = \pi(t)\frac{dS(t)}{S(t)} + (X(t)-\pi(t))\frac{dB(t)}{B(t)} + d R(t).
\end{align}

\par We assume that the processes and random variables $N_1(t)$, $\{U_i\}_{i\geq1}$, $N_2(t)$, $\{V_i\}_{i\geq1}$ and $W_0(t)$ are mutually independent under the probability $P$. Let $\mathcal{F}_t$ be the completion of the $\sigma$-field generated by $N_1(dt,dz)$, $N_2(dt,dz)$, $W_0(t)$, $\{U_i\}$ and $\{V_i\}$ under $P$ and $\mathbb{F}:=\{ \mathcal{F}_t | t \in [0,T] \}$. By Proposition 3 of \cite{kabanov1979absolute}, the compensated random measures in probability space $(\Omega, \mathbb{F}, P)$ are given by
\begin{align*}
\widetilde{N}_1(ds,dz) = N_1(ds,dz) - \lambda(s)F_1(dz)ds,
\end{align*}
and
\begin{align*}
\widetilde{N}_2(ds,dz) = N_2(ds,dz) - \rho F_2(dz)ds,
\end{align*}
where $F_1$ and $F_2$ are distributions of $U_i$ and $V_i$, respectively. For notational simplicity, we also denote
\begin{align}
\nonumber &\mu_1= \int_{\mathbb{R}_{>0}} z F_1(dz), \quad &\sigma_1^2 = \int_{\mathbb{R}_{>0}} z^2 F_1(dz), \\
\label{mmvSN.moment} &\mu_2= \int_{\mathbb{R}_{>0}} z F_2(dz), \quad &\sigma_2^2 = \int_{\mathbb{R}_{>0}} z^2 F_2(dz).
\end{align}
Thus, the differential equations for $\lambda(t)$ and $R(t)$ can be written as
\begin{align*}
d\lambda(t) = (-\delta \lambda(t) + \rho \mu_2) dt + \int_{\mathbb{R}_{>0}} z \widetilde{N}_2(dt,dz),
\end{align*}
and
\begin{align}
\nonumber dR(t) =& \bigg(\kappa_r u(t) - (\kappa_r - \kappa)\bigg)\mu_1\lambda(t) dt -  \int_{\mathbb{R}_{>0}} z u(t) \widetilde{N}_1(dt,dz) \\
\label{mmvSN.surplus.0} &+ \bigg(\iota_r v(t) - (\iota_r - \iota)\bigg)k \mu_2\rho dt -  k\int_{\mathbb{R}_{>0}} z v(t) \widetilde{N}_2(dt,dz).
\end{align}
By substituting \eqref{mmvSN.surplus.0} into \eqref{mmvSN.welath.0}, the controlled surplus process of the insurer satisfies the following stochastic differential equation,
\begin{align*}
dX(t) =  \bigg( r X(t)+\pi(t) (\mu_0 - r)  + (\kappa_ru(t) - \kappa_r + \kappa)\mu_1\lambda(t)  + (\iota_r v(t) - \iota_r + \iota) k\mu_2\rho\bigg) dt \\
+ \pi(t)  \sigma_0  dW_0(t)  -  \int_{\mathbb{R}_{>0}} z u(t) \widetilde{N}_1(dt,dz) -  k \int_{\mathbb{R}_{>0}} z v(t) \widetilde{N}_2(dt,dz),
\end{align*}
where $(\pi(t), u(t), v(t))$ are the control variables.

\subsection{Monotone Mean-Variance Criterion}

\par The insurer aims to choose some admissible strategies $(\pi(t), u(t), v(t))$ to maximize the terminal surplus under MMV criterion introduced by \cite{maccheroni2009portfolio}, that is,
\begin{equation}
\label{obj.0} V_{\theta}(X(T)) = \min_{Q \in \Delta^2(P)} \left \{ \mathbb{E}^Q[X(T)]+\frac{1}{2 \theta} C(Q||P) \right \},
\end{equation}
where
\begin{equation*}
\Delta^2(P) = \left\{Q \ll P: Q( \Omega ) = 1, E\left[\left(\frac{dQ}{dP}\right)^2\right] < \infty \right\},
\end{equation*}
and $C(Q||P)$ is a penalty function satisfying
\begin{equation*}
C(Q||P) = \mathbb{E}^P \left[\left(\frac{dQ}{dP}\right)^2 \right] - 1,
\end{equation*}
and $\theta$ is an index to measure the risk aversion of the insurer. 
If the alternative probability measure $Q$ is constrained in a subset of $\Delta^2(P)$:
\begin{equation*}
\widetilde{\Delta}^2(P) = \bigg\{Q \sim P: Q( \Omega ) = 1, \mathbb{E}^P \bigg[\left(\frac{dQ}{dP}\right)^2\bigg] < \infty \bigg\},
\end{equation*}
by applying the exponential martingale representation, we know that $\frac{dQ}{dP}|_{\mathcal{F}_t}$ is the solution of the SDE:
\begin{equation}
\label{mmvSN.Q.representation}
Y(t) = 1  + \int_0^t Y(s) o(s)dW_0(s) + \int_0^t \int_{\mathbb{R}_{>0}} Y(s-) p(s,z) \widetilde{N}_1(ds,dz) + \int_0^t \int_{\mathbb{R}_{>0}} Y(s-) q(s,z) \widetilde{N}_2(ds,dz),
\end{equation}
where $(o(t), p(t,z), q(t,z))$ are some suitable process\footnote{$o$ is an $\mathbb{F}$-adapted process, and for any fixed $z \in \mathbb{R}_{>0}$, $p$ and $q$ are $\mathbb{F}$-predictable processes. All of them should satisfy some integrable conditions such as Novikov's condition.} depending on the selection of $Q$ (see \cite{mataramvura2008risk} and \cite{Elliott2009Portfolio} for more details). Specially, we can regard $o(t)$, $p(t,z)$ and $q(t,z)$ as control processes which have a one-to-one correspondence with $Q \in \widetilde{\Delta}^2(P)$. 

\par In this paper, we do not constrain the chosen of $Q$ in $\widetilde{\Delta}^2(P)$, hence $\frac{dQ}{dP}\big|_{\mathcal{F}_t}$ is not necessarily positive almost surely for $t\in[0,T]$, and is not necessarily a standard exponential martingale. Fortunately, by using the tools of discontinuous martingale analysis in \cite{kabanov1979absolute}, \cite{liptser2012theory} and \cite{hansen2006robust}, it can be proved that the exponential martingale representation \eqref{mmvSN.Q.representation} still holds and $\frac{dQ}{dP}|_{\mathcal{F}_t}$ is a solution of \eqref{mmvSN.Q.representation} (see \cite{li2022Monotone} for more details). In this case, $o(t)$, $p(t,z)$ and $q(t,z)$ can explode at the time that $Y$ hits zero, but for any
\begin{equation*}
\zeta_n := \inf\left\{ t \geq 0; Y(t) \leq \frac{1}{n}\right\},
\end{equation*}
$o(t)$, $p(t,z)$ and $q(t,z)$ should satisfy the following integrable condition for any integer $n \geq 1$:
\begin{align}
\nonumber E \bigg[ \int_0^{T \wedge \zeta_n} o(t)^2 dt + \int_0^{T \wedge \zeta_n} \int_{\mathbb{R}_{>0}} (1 - \sqrt{p(t,z)+1})^2 \lambda(t)F_1(dz) dt \\
\label{mmvSN.condition.integrable.property} + \int_0^{T \wedge \zeta_n} \int_{\mathbb{R}_{>0}} (1 - \sqrt{q(t,z)+1})^2 \rho F_2(dz) dt \bigg] < \infty.
\end{align}
We give the following example to illustrate this more clearly.

\begin{example}
	Let the intensity of $N_2$, $\rho$ = 1 and define
	\begin{align*}
	Y(T) = \frac{N_2(\tau)}{\tau}, \text{ for some } 0 < \tau < T,
	\end{align*}
	and an alternative measure
	\begin{align*}
	Q(A) := \int_A Y(T) dP, \text{ for all } A \in \mathcal{F}_T.
	\end{align*}
	Since $Y(T)$ is non-negative and square-integrable, and
	\begin{align*}
	Q(\Omega) = \mathbb{E} Y(T) =& \mathbb{E} \bigg[ \frac{N_2(\tau)}{\tau} \bigg] = 1,
	\end{align*}
	we have $Q \in \Delta^2(P)$.
	In this case,
	\begin{align*}
	Y(t) = \mathbb{E} \big[ Y(T) \big| \mathcal{F}_t \big] = 
	\begin{cases}
	&\frac{N_2(t) - t +\tau}{\tau}, \quad 0 \leq t < \tau, \\
	&\frac{N_2(\tau)}{\tau}, \quad \tau \leq t \leq T.
	\end{cases}
	\end{align*}
	Obviously, $Y(t)$ is not almost surely positive for $\tau \leq t \leq T$. Moreover, we have
	\begin{align*}
	Y(t) =  1 + \int_0^{t \vee \tau} \frac{1}{\tau} d\widetilde{N}_2(s) = 1 + \int_0^{t} \frac{1}{\tau} 1_{\{s \leq \tau\}}  d\widetilde{N}_2(s).
	\end{align*}
	Define
	\begin{align*}
	\eta(t) = \frac{1}{\tau} Y(t)^{-1}1_{\{t < \tau\}} =
	\begin{cases}
	&\frac{1}{N_2(t)-t+\tau}, \quad 0 \leq t < \tau, \\
	&0, \quad \tau \leq t \leq T,
	\end{cases}
	\end{align*}
	then
	\begin{align*}
	Y(t) =& 1 + \int_0^{t} \frac{1}{\tau} 1_{\{s \leq \tau\}}  d\widetilde{N_2}(s) = 1 + \int_0^t Y(s-) \eta(s-) d\widetilde{N_2}(s).
	\end{align*}
	We have $q(t) = \eta(t-)$.
	Note that on the set $\{\omega:N_2(\tau)=0\}$, we have
	\begin{align*}
		\int_0^{T} (1 - \sqrt{q(t)+1})^2 dt  =& \int_0^{\tau} \left(1 - \sqrt{\frac{1}{N_2(t-)-t+\tau}+1}\right)^2 dt \\
		=& \int_0^{\tau} \left(1 - \sqrt{\frac{1}{\tau-t}+1}\right)^2 dt = \infty.
	\end{align*}
	But for any $\zeta_n < \tau$, it satisfies the condition \eqref{mmvSN.condition.integrable.property} that
	\begin{align*}
		E \bigg[\int_0^{T \wedge \zeta_n} (1 - \sqrt{q(t)+1})^2 dt \bigg] = E \bigg[\int_0^{\tau \wedge \zeta_n} \left(1 - \sqrt{\frac{1}{N_2(t-)-t+\tau}+1}\right)^2 dt\bigg] < \infty,
	\end{align*}
\end{example}

\subsection{Hamilton-Jacobi-Bellman-Isaacs Equation}

Based on the above discussion, we now summarize the dynamics of $X(t)$, $Y(t)$ and $\lambda(t)$ below
\begin{subequations}
	\begin{empheq} [left=\empheqlbrace] {align}
	\nonumber dX(t) =&  \bigg( r X(t)+\pi(t)( \mu_0 - r)  + (\kappa_ru(t) - \kappa_r + \kappa)\mu_1\lambda(t)  + (\iota_r v(t) - \iota_r + \iota)k\mu_2\rho\bigg) dt \\
	\label{mmvSN.dynamic.x} &+ \pi(t)  \sigma_0  dW_0(t)  -  \int_{\mathbb{R}_{>0}} z u(t) \widetilde{N}_1(dt,dz) - k \int_{\mathbb{R}_{>0}} z v(t) \widetilde{N}_2(dt,dz), \\
	\label{mmvSN.dynamic.y} dY(t) =& Y(t) o(t) dW_0(t) + \int_{\mathbb{R}_{>0}} Y(t-) p(t,z) \widetilde{N}_1(dt,dz) + \int_{\mathbb{R}_{>0}} Y(t-) q(t,z) \widetilde{N}_2(dt,dz), \\
	\label{mmvSN.dynamic.lambda} d\lambda(t) =& (-\delta \lambda(t) + \rho \mu_2) dt + \int_{\mathbb{R}_{>0}} z \widetilde{N}_2(dt,dz).
	\end{empheq}
\end{subequations}
and the original maximization problem is transformed to an auxiliary two-player zero-sum game as follows,
\begin{problem} Let
	\begin{equation}
	\label{obj.3}
	J^{a,b}(s,x,y,\lambda) = \mathbb{E}_{s,x,y,\lambda}^P \bigg[X^{a}(T)Y^{b}(T)+\frac{1}{2\theta}(Y^{b}(T))^2 \bigg],
	\end{equation}
	where $\mathbb{E}_{s,x,y,\lambda}^P[\cdot]$ represents $\mathbb{E}^P[\cdot|X^a(s) = x,Y^{b}(s) = y,\lambda(s) = \lambda]$. Player one wants to maximize $J^{a,b}(s,x,y,\lambda)$ with strategy $a = (\pi,u,v)$ over $\mathcal{A}[s,T]$ defined below and player two wants to maximize $-J^{a,b}(s,x,y,\lambda)$ with strategy $b = (o,p,q)$ over $\mathcal{B}[s,T]$ defined below.
\end{problem}

Notice that the starting time of Problem \eqref{obj.0} is fixed at zero, and Problem \eqref{obj.3} extended Problem \eqref{obj.0} by letting the initial time become $s$ (see \cite{li2021optimal,li2022Monotone} for more details). If there exist the optimal strategies $a^*$ and $b^*$ for Problem \eqref{obj.3}, then
\begin{equation*}
V_{\theta}(X(T)) = J^{a,b^*}(0,x,1,\lambda).
\end{equation*}

\begin{definition} \label{mmvSN.definition.admissible}
	\par The strategy $\{a(t)\}_{0\leq t\leq T} = \{(\pi(t),u(t),v(t))\}_{0\leq t\leq T}$ employed by player one is admissible if the following conditions are met: $\pi: [s,T] \to \mathbb{R}$ is an $\mathbb{F}$-adapted process, and $u: [s,T] \to \mathbb{R}$ and $v: [s,T] \to \mathbb{R}$ are $\mathbb{F}$-predictable processes such that \eqref{mmvSN.dynamic.x} is well-defined and satisfies $\mathbb{E}^PX^u(t) < \infty$. We denote the set of all admissible strategies $a(t)$ as $\mathcal{A}[s,T]$.
	
	\par The strategy $\{b(t,z)\}_{0\leq t\leq T} = \{(o(t),p(t,z),q(t,z))\}_{0\leq t\leq T}$ used by player two is admissible if the following conditions are met: $o: [s,T] \to \mathbb{R}$ is an $\mathbb{F}$-adapted process, and for any fixed $z \in \mathbb{R}_{>0}$, $p(\cdot,z): [s,T] \to \mathbb{R}$ and $q(\cdot,z): [s,T] \to \mathbb{R}$ are $\mathbb{F}$-predictable processes. These processes should satisfy the condition:
	\begin{equation*}
		\Delta L(t) \equiv \int_{\mathbb{R}{>0}} p(t,z) N_1({t},dz) + \int{\mathbb{R}_{>0}} p(t,z) N_2({t},dz) \geq -1, \quad t \in [s,T],
	\end{equation*}
	Additionally, the stochastic differential equation \eqref{mmvSN.dynamic.y} should have a unique solution, which is a nonnegative $\mathbb{F}$-adapted square integrable $P$-martingale satisfying $\mathbb{E}^PY(t) = 1$ for $t \in [s,T]$. We use $\mathcal{B}[s,T]$ to represent the set of all admissible strategies $b(t,z)$.
\end{definition}

We present the following verification theorem, whose proof closely follows the one of Theorem 2.2.2 in \cite{yeung2006cooperative} and Theorem 3.2 in \cite{mataramvura2008risk}:
\begin{theorem} \label{mmvSN.theorem.verification.0}
	{\bf (Verification Theorem)} Suppose that $W: (s,x,y,\lambda) \mapsto[s,T] \times \mathbb{R} \times \mathbb{R}_{\geq 0} \times \mathbb{R}_{>0}$ is a $C^{1,2,2,2}$ function satisfying the following condition
	\begin{align}
	\label{mmvSN.hjbi.equation}
	\begin{cases}
	&\mathcal{L}_t^{a^*,b^*}W(t,x,y,\lambda) = 0, \qquad \forall (t,x,y,\lambda) \in [s,T) \times \mathbb{R} \times \mathbb{R}_{\geq 0} \times \mathbb{R}_{>0}, \\
	&\mathcal{L}_t^{a^*,b}W(t,x,y,\lambda) \geq 0, \qquad \forall b \in \mathcal{B}[s,T], \qquad \forall (t,x,y,\lambda) \in [s,T) \times \mathbb{R} \times \mathbb{R}_{\geq 0} \times \mathbb{R}_{>0}, \\
	&\mathcal{L}_t^{a,b^*}W(t,x,y,\lambda) \leq 0, \qquad \forall a \in \mathcal{A}[s,T], \qquad \forall (t,x,y,\lambda) \in [s,T) \times \mathbb{R} \times \mathbb{R}_{\geq 0} \times \mathbb{R}_{>0}, \\
	&W(T,x,y,\lambda) = xy + \frac{1}{2 \theta} y^2, \qquad \forall (x,y,\lambda) \in \mathbb{R} \times \mathbb{R}_{\geq 0} \times \mathbb{R}_{>0}, 
	\end{cases}
	\end{align}
	where $\mathcal{L}_t^{a,b}$ is the infinitesimal generator of $(X(t),Y(t),\lambda(t))$ given by
	\begin{align}
	\nonumber \mathcal{L}_t^{a,b} W(t,x,y,\lambda) =& W_t + \biggl\{ r x+\pi (\mu_0 -  r ) + (\kappa_ru - \kappa_r + \kappa)\mu_1\lambda  \\
	\nonumber &+ (\iota_r v - \iota_r + \iota)k\mu_2\rho\biggr\} W_x  +(-\delta \lambda + \rho \mu_2) W_\lambda + \frac{1}{2} \pi^2  \sigma_0 ^2 W_{xx} + \frac{1}{2} y^2 o^2 W_{yy} + \pi  \sigma_0  y o W_{xy}\\
	\nonumber &+ \lambda \int_{\mathbb{R}_{>0}} \{ W(t,x-u z,y+yp(z),\lambda) - W(t,x,y,\lambda) + u z W_x - y p(z) W_y \} F_1(dz) \\
	\label{mmvSN.generator}& + \rho \int_{\mathbb{R}_{>0}} \{ W(t,x-k v z,y+yq(z),\lambda+z) - W(t,x,y,\lambda) + k v z W_x - yq(z) W_y - z W_\lambda \} F_2(dz).
	\end{align}
	Then, 
	\begin{align*}
	W(s,x,y,\lambda) = \sup_{a \in \mathcal{A}[s,T]} \inf_{b \in \mathcal{B}[s,T]} J^{a,b}(s,x,y,\lambda).
	\end{align*}
\end{theorem}
The infinitesimal generator is derived by using the piece-wise deterministic processes theory in \cite{davis1984piecewise-deterministic}. For more details about the infinitesimal generator of Cox process driven by shot noise intensity, see \cite{dassios1987insurance}, \cite{dassios1989martingales}, \cite{dassios2003pricing} and \cite{jang2003shot}.

\section{Value Function and Optimal Strategies} \label{mmvSN.section.main_result}

In this section, we shall first build a candidate solution to the HJBI equation \eqref{mmvSN.hjbi.equation} based on Proposition \ref{mmvSN.lemma.solution.suppose}. Then We shall conduct an analysis of the optimal strategies. Subsequently, we shall demonstrate that the constructed candidate solution indeed represents the value function of Problem \eqref{obj.0}, and we will also provide its explicit form. At the end of this section, we shall present the efficient frontier for our problem. For the sake of clarity, we defer the proofs to Appendix A.

We show the main result for the optimal strategies and the value function as follows:
\begin{theorem} \label{mmvSN.main_result}
	\par Assume the initial values $X(0) = x_0$, $Y(0) = 1$ and $\lambda(0) = \lambda_0$. The value function for Problem \eqref{obj.0} is given by
	\begin{align*}
	W(t,x) = -\frac{\theta(e^{r(T-t)}x+\alpha(t)\lambda(t) + \beta(t))^2}{4e^{\eta(t) \lambda(t) + \zeta(t)}} + \frac{\theta(e^{rT}x_0+\frac{2}{\theta}e^{\eta(0) \lambda_0 + \zeta(0)}+\alpha(0)\lambda_0 + \beta(0))^2}{4e^{\eta(t) \lambda(t) + \zeta(t)}},
	\end{align*}
	and the optimal feedback strategies are given as follows
	\begin{align*}
	\begin{cases}
	\pi^*(t,x) =& \frac{ \mu_0 -  r  }{ \sigma_0 ^2} \bigg( - x  + x_0 e ^{rt} + \frac{1}{\theta} e^{\eta(0)\lambda_0 + \zeta(0) - r(T-t)} - e^{-r(T-t)}[\alpha(t)\lambda(t)+\beta(t) - \alpha(0)\lambda_0 - \beta(0)] \bigg), \\
	u^*(t,x) =& \frac{\kappa_r \mu_1 }{ \sigma_1^2}  \bigg( - x  + x_0 e ^{rt} + \frac{1}{\theta} e^{\eta(0)\lambda_0 + \zeta(0) - r(T-t)} - e^{-r(T-t)}[\alpha(t)\lambda(t)+\beta(t) - \alpha(0)\lambda_0 - \beta(0)] \bigg), \\
	v^*(t,x) =& \frac{1}{ k } \phi(t) \bigg( - x  + x_0 e ^{rt} + \frac{1}{\theta} e^{\eta(0)\lambda_0 + \zeta(0) - r(T-t)} - e^{-r(T-t)}[\alpha(t)\lambda(t)+\beta(t) - \alpha(0)\lambda_0 - \beta(0)] \bigg)  \\
	&+ \frac{\alpha(t)}{k} e ^{-r(T-t)},
	\end{cases}
	\end{align*}	
	where
	\begin{align*}
	\begin{cases}
	\eta(t) =& \frac{\kappa_r^2   \mu_1^2}{\delta \sigma_1^2}(1-e^{-\delta(T-t)}), \\
	\zeta(t) =& \rho \int_t^T  \biggl\{ \phi(s)^2 \int_{\mathbb{R}_{>0}}  z^2 e^{-\eta(s) z} F_2(dz) - \int_{\mathbb{R}_{>0}}  e^{-\eta(s) z} F_2(dz)  \biggr\} ds + (\rho  + \frac{ ( \mu_0 -  r )^2  }{    \sigma_0 ^2})(T-t), \\
	\alpha(t) =& \frac{(\kappa_r - \kappa)\mu_1}{\delta + r} (e^{- \delta(T-t)} - e^{r(T-t)}), \\
	\beta(t) =& \frac{\rho (\kappa_r - \kappa)  (\iota_r + 1) \mu_1 \mu_2}{\delta + r} \bigg( \frac{1}{\delta}(1-e^{-\delta(T-t)})+\frac{1}{r}(1-e^{r(T-t)}) \bigg)  - \frac{\rho ( \iota_r - \iota) k\mu_2}{r} (1-e^{r(T-t)}),
	\end{cases}
	\end{align*}
	and $\phi(t) = \frac{ (\iota_r+1) \mu_2 -  \int_{\mathbb{R}_{>0}} z e^{-\eta(t) z} F_2(dz)}{\int_{\mathbb{R}_{>0}} z^2 e^{-\eta(t) z} F_2(dz)} $.
\end{theorem}

To achieve the aforementioned result, we employ a three-step approach. Firstly, we tackle Problem \eqref{obj.3}, and the outcomes shall be presented in Proposition \ref{mmvSN.lemma.solution.suppose}. Problem \eqref{obj.3} involves three processes, namely $X(t)$, $Y(t)$, and $\lambda(t)$. It is worthnoting that the insurer can only observe the surplus process $X(t)$ and the intensity process $\lambda(t)$. As $Y(t)$ serves as an auxiliary process introduced to facilitate problem-solving, it is unobservable. However, Proposition \ref{mmvSN.lemma.solution.suppose} provides expressions of the value function and the optimal strategies as functions of the auxiliary process $Y(t)$. To remove the dependency on $Y(t)$, we take the second step that establish a relationship between the processes $X(t)$ and $Y(t)$(see Lemma \ref{mmvSN.lemma.relationship_of_XY}). Finally, in the third step, we can represent the value function and optimal strategies solely as functions of $X(t)$ and $\lambda(t)$ (see Proposition \ref{mmvSN.proposition.optimal_function}).

\begin{proposition} \label{mmvSN.lemma.solution.suppose}
	
	For any $(t,\lambda) \in [0,T]\times \mathbb{R}_{>0}$, if there exist sufficiently smooth functions $G(t)$, $H(t,\lambda)(>0)$, $I(t,\lambda)$ and $K(t,\lambda)$ satisfying the following differential equations, respectively\footnote{For any appropriate function $\psi(t,\lambda)$, define $\psi_z : = \psi(t,\lambda+z)-\psi(t,\lambda)$ as its finite difference on the spatial variable.}
	\begin{align}
	\label{mmvSN.equation.function.value.coefficients}
	\begin{cases}
	0 =& G_t +  r G(t), \\
	0 = &H_t -\delta \lambda H_{\lambda} +  \rho \int_{\mathbb{R}_{>0}} H_z F_2(dz) + \frac{H(t,\lambda) ( \mu_0 -  r )^2 }{    \sigma_0 ^2} + \frac{H(t,\lambda)  \lambda \kappa_r^2  \mu_1^2 }{ \sigma_1^2} \\
	&+ \frac{\rho \bigg(\iota_r  \mu_2 + 2 {\bf \widetilde{H}}(H_z) \bigg)^2 }{ 2 {\bf \widetilde{H}}(z)}  - 2 \rho  {\bf \widetilde{H}}(\frac{H_z^2}{z}), \\
	0 =& I_t -\delta \lambda I_{\lambda} +  \rho \int_{\mathbb{R}_{>0}} I_z F_2(dz) + \bigg( ( - \kappa_r + \kappa)\mu_1\lambda  + ( - \iota_r + \iota) k \mu_2\rho\bigg) G(t) \\
	&+ \frac{  \rho \iota_r \mu_2 {\bf \widetilde{H}}(I_z) }{ {\bf \widetilde{H}}(z)} - \rho \biggl\{ 2 {\bf \widetilde{H}}(\frac{H_z I_z}{z}) - 2 \frac{{\bf \widetilde{H}}(H_z) {\bf \widetilde{H}}(I_z)}{{\bf \widetilde{H}}(z)} \biggr\} , \\
	0 =& K_t -\delta \lambda K_{\lambda} +  \rho \int_{\mathbb{R}_{>0}} K_z F_2(dz) - \frac{\rho}{2} \biggl\{ {\bf \widetilde{H}}(\frac{I_z^2}{z}) - \frac{{\bf \widetilde{H}}(I_z)^2}{{\bf \widetilde{H}}(z)} \biggr\} ,
	\end{cases}
	\end{align}
	with boundry conditions $G(T) = 1$, $H(T,\lambda) = \frac{1}{2\theta}$, $I(T,\lambda) = K(T,\lambda) = 0$, where
	\begin{align*}
	{\bf \widetilde{H}}(f) := \int_{\mathbb{R}_{>0}} \frac{f(z) z}{2 H(t,\lambda+z)}F_2(dz),
	\end{align*}
	for all integrable function $f(z)$, then the solution to the HJB equation \eqref{mmvSN.hjbi.equation} is
	\begin{equation}
	\label{mmvSN.value.function.W}
	W(t,x,y,\lambda) = G(t)xy + H(t,\lambda)y^2 +I(t,\lambda)y + K(t,\lambda).
	\end{equation}
	Moreover, the corresponding optimal feedback strategies are given by
	\begin{numcases}{}
	\nonumber
	\pi^*(t,x,y,\lambda) = \frac{2H(t,\lambda) ( \mu_0 -  r ) y }{ G(t)    \sigma_0 ^2}, \\
	\nonumber
	u^*(t,x,y,\lambda)= \frac{2H(t,\lambda)\kappa_r \mu_1 y}{ G(t)   \sigma_1^2},  \\
	\nonumber
	v^*(t,x,y,\lambda) = \frac{1}{kG(t)}\biggl\{\frac{\iota_r  \mu_2 y }{ {\bf \widetilde{H}}(z)} + \frac{2y{\bf \widetilde{H}}(H_z) }{ {\bf \widetilde{H}}(z)} + \frac{{\bf \widetilde{H}}(I_z)}{ {\bf \widetilde{H}}(z)}\biggr\}, \\
	\label{mmvSN.strategy.saddlepoint.equation}
	o^*(t,z,x,y,\lambda) = - \frac{  \mu_0 -  r   }{\sigma_0}, \\
	\nonumber
	p^*(t,z,x,y,\lambda)= \frac{ \kappa_r \mu_1 z}{\sigma_1^2},  \\
	\nonumber
	q^*(t,z,x,y,\lambda)= -\frac{1 }{2H(t,\lambda+z)y} \biggl\{\bigg(2  ({\bf 1}- \frac{z{\bf \widetilde{H}} }{ {\bf \widetilde{H}}(z)})(H_z) -  \frac{\iota_r   \mu_2  z}{ {\bf \widetilde{H}}(z)}\bigg)y +  ({\bf 1}- \frac{z{\bf \widetilde{H}} }{ {\bf \widetilde{H}}(z)})(I_z)\biggr\}.
	\end{numcases}
\end{proposition}

We observe that the optimal feedback strategies \eqref{mmvSN.strategy.saddlepoint.equation} only depend on the auxiliary process $Y(t)$ rather than the surplus process $X(t)$. However $Y(t)$ is not observable in practice. To address this problem, we shall introduce a lemma which shows the relationship of $(X(t),Y(t))$ and $(X(s),Y(s))$ for any $0 \leq s \leq t \leq T$. Then we can rewrite the optimal feedbacks $(\pi,u,v)$ as functions that only depend on the current value of surplus process $X(t)$, and the initial values $(X(s),Y(s))$ at time $s$. Here is the lemma:

\begin{lemma} \label{mmvSN.lemma.relationship_of_XY}
	\par Let $X^*(t)$ and $Y^*(t)$ be the state processes under the feedback strategies defined by \eqref{mmvSN.strategy.saddlepoint.equation}, then for any $0 \leq s \leq t \leq T$, we have
	\begin{align}
	\label{mmvSN.equation.relationship_of_XY} X^*(t) G(t) - X^*(s) G(s)= - 2   Y^*(t) H(t,\lambda(t)) + 2 Y^*(s) H(s,\lambda(s))  -  I(t,\lambda(t)) + I(s,\lambda(s)).
	\end{align}
\end{lemma} 
Note that the process $\lambda(t)$ does not depend on the controls, we can easily calculate its value at any time $t \in [0,T]$. For simplicity, in the following context, we denote $W(t,x,y):=W(t,x,y,\lambda(t))$ and $W(t,x):=W(t,x,\lambda(t))$. By substituting \eqref{mmvSN.equation.relationship_of_XY} into \eqref{mmvSN.value.function.W} and \eqref{mmvSN.strategy.saddlepoint.equation}, we can rewrite the value function and the optimal strategies as follows:
\begin{proposition} \label{mmvSN.proposition.optimal_function}
	\par For any $0 \leq s \leq t \leq T$, let $X^*(t)$ and $Y^*(t)$ be the state processes under the feedback strategies defined by \eqref{mmvSN.strategy.saddlepoint.equation} and let $X^*(s)$ and $Y^*(s)$ be the initial values at time $s$, then the value function is given by
	\begin{align}
	\label{mmvSN.value_function.s}
	W(t,X^*(t)) = -\frac{(G(t)X^*(t)+I(t,\lambda(t)))^2}{4H(t,\lambda(t))} + \frac{(G(s)X^*(s)+2H(s,\lambda(s)Y^*(s)+I(s,\lambda(s)))^2}{4H(t,\lambda(t))}+K(t,\lambda(t)),
	\end{align}
	and the optimal strategies $\pi^*$, $u^*$ and $v^*$ at time $t$ are given by:
	\begin{align}
	\label{mmvSN.strategy.saddlepoint.equation.precommitted}
	\begin{cases}
	\pi^*(t) =& \frac{ \mu_0 -  r  }{ \sigma_0 ^2} \bigg( - X^*(t)  + X^*(s) \frac{G(s)}{G(t)} + 2 Y^*(s) \frac{H(s,\lambda(s))}{G(t)}  -  \frac{I(t,\lambda(t))}{G(t)} + \frac{I(s,\lambda(s))}{G(t)} \bigg), \\
	u^*(t) =& \frac{\kappa_r \mu_1 }{ \sigma_1^2}  \bigg( - X^*(t)  + X^*(s) \frac{G(s)}{G(t)} + 2 Y^*(s) \frac{H(s,\lambda(s))}{G(t)}  -  \frac{I(t,\lambda(t))}{G(t)} + \frac{I(s,\lambda(s))}{G(t)} \bigg), \\
	v^*(t) =& \frac{1}{2 k H(t,\lambda(t))}\bigg(\frac{\iota_r  \mu_2 }{  {\bf \widetilde{H}}(z)} + \frac{2 {\bf \widetilde{H}}(H_z) }{  {\bf \widetilde{H}}(z)}\bigg)  \bigg( - X^*(t)  + X^*(s) \frac{G(s)}{G(t)} + 2 Y^*(s) \frac{H(s,\lambda(s))}{G(t)}  -  \frac{I(t,\lambda(t))}{G(t)} + \frac{I(s,\lambda(s))}{G(t)} \bigg)  \\
	&+ \frac{1}{kG(t)}\frac{{\bf \widetilde{H}}(I_z)}{ {\bf \widetilde{H}}(z)}.
	\end{cases}
	\end{align}
\end{proposition}

It's important to note that \eqref{mmvSN.value_function.s} and \eqref{mmvSN.strategy.saddlepoint.equation.precommitted} also rely on $Y^*(s)$. Keep in mind that $s$ is chosen as the initial time before time $t$, implying that $Y^*(s)$ is, in fact, known. In conclusion, by setting the initial time of Problem \eqref{obj.3} to be zero, and the initial values of Problem \eqref{obj.3} to be $X^*(0)=x_0$, $Y^*(0)=1$, and $\lambda(0)=\lambda_0$, we ultimately obtain the solution to the original Problem \eqref{obj.0} (see Theorem \ref{mmvSN.main_result}). Furthermore, the optimal feedbacks in \eqref{mmvSN.strategy.saddlepoint.equation} solely rely on the current value of $Y^*(t)=y$ at time $t$, making them time-consistent. However, as illustrated in \eqref{mmvSN.strategy.saddlepoint.equation.precommitted}, the strategies depend not only on $X^*(t)$ but also on the values $X^*(s)$ and $Y^*(s)$ at some time point $s$ in the past. Consequently, the optimal strategies presented in the form of \eqref{mmvSN.strategy.saddlepoint.equation.precommitted} are precommitted.

Now, we give the explicit expressions of $G(t)$, $H(t,\lambda)$, $I(t,\lambda)$ and $K(t,\lambda)$ by the following lemma.
\begin{lemma} \label{mmvSN.lemma.coefficients}
	\par The solutions to \eqref{mmvSN.equation.function.value.coefficients} are given as follows
	\begin{align}
	\label{mmvSN.equation.function.value.coefficients.solution}
	\begin{array}{l}
	G(t) = e^{r(T-t)}, \quad H(t,\lambda) = \frac{1}{2 \theta}e^{\eta(t) \lambda + \zeta(t)}, \\
	I(t,\lambda) = \alpha(t) \lambda + \beta(t), \quad K(t,\lambda) = 0,
	\end{array}
	\end{align}
	where
	\begin{align*}
	\begin{cases}
	\eta(t) =& \frac{\kappa_r^2   \mu_1^2}{\delta \sigma_1^2}(1-e^{-\delta(T-t)}), \\
	\zeta(t) =& \rho \int_t^T  \biggl\{ \phi(s)^2 \int_{\mathbb{R}_{>0}}  z^2 e^{-\eta(s) z} F_2(dz) - \int_{\mathbb{R}_{>0}}  e^{-\eta(s) z} F_2(dz)  \biggr\} ds + (\rho  + \frac{ ( \mu_0 -  r )^2  }{    \sigma_0 ^2})(T-t), \\
	\alpha(t) =& \frac{(\kappa_r - \kappa)\mu_1}{\delta + r} (e^{- \delta(T-t)} - e^{r(T-t)}), \\
	\beta(t) =& \frac{\rho (\kappa_r - \kappa)  (\iota_r + 1) \mu_1 \mu_2}{\delta + r} \bigg( \frac{1}{\delta}(1-e^{-\delta(T-t)})+\frac{1}{r}(1-e^{r(T-t)}) \bigg)  - \frac{\rho ( \iota_r - \iota) k\mu_2}{r} (1-e^{r(T-t)}),
	\end{cases}
	\end{align*}
	and $\phi(t) = \frac{ (\iota_r+1) \mu_2 -  \int_{\mathbb{R}_{>0}} z e^{-\eta(t) z} F_2(dz)}{\int_{\mathbb{R}_{>0}} z^2 e^{-\eta(t) z} F_2(dz)} $.
\end{lemma}

\par Based on the above analysis and preparation, we are able to solve the HJBI equation \eqref{mmvSN.hjbi.equation}. Hence, we can use the verification theorem to prove that the candidate function and the corresponding strategies are indeed the value function and the optimal strategies.
\begin{theorem} \label{mmvSN.theorem.solution.explicit}
	\par {\bf (Value function and optimal strategy)} Assume the initial values $X(s) = x_s$ and $Y(s) = y_s$. The value function for Problem \eqref{obj.3} is given by $W(t,x) \in C^{1,2}([s,T] \times \mathbb{R})$,
	\begin{align}
	\label{mmvSN.function.value} W(t,x;s,x_s,y_s) = -\frac{\theta(e^{r(T-t)}x+\alpha(t)\lambda(t) + \beta(t))^2}{4e^{\eta(t) \lambda(t) + \zeta(t)}} + \frac{\theta(e^{rT}x_s+\frac{2 y_s}{\theta}e^{\eta(s) \lambda(s) + \zeta(s)}+\alpha(s)\lambda(s) + \beta(s))^2}{4e^{\eta(t) \lambda(t) + \zeta(t)}},
	\end{align}
	and the equilibrium strategies are the precommitted feedback strategies
	\begin{align}
	\label{mmvSN.strategy.optimal.a} 
	\begin{cases}
	\pi^*(t,x;s,x_s,y_s) =& \frac{ \mu_0 -  r  }{ \sigma_0 ^2} \bigg( - x  + x_s e^{r(t-s)} +  \frac{y_s}{\theta} e^{\eta(t) \lambda(t) + \zeta(t)} e^{-r(T-t)}  \\
	&-  (\alpha(t) \lambda(t) + \beta(t) )e^{-r(T-t)}  + (\alpha(s) \lambda(s) + \beta(s) )e^{-r(T-t)} \bigg), \\
	u^*(t,x;s,x_s,y_s)=& \frac{\kappa_r \mu_1 }{ \sigma_1^2}  \bigg( - x  + x_s e^{r(t-s)} +  \frac{y_s}{\theta} e^{\eta(t) \lambda(t) + \zeta(t)} e^{-r(T-t)}  \\
	&-  (\alpha(t) \lambda(t) + \beta(t) )e^{-r(T-t)}  + (\alpha(s) \lambda(s) + \beta(s) )e^{-r(T-t)} \bigg), \\
	v^*(t,x;s,x_s,y_s) =& \frac{1}{ k } \phi(t)  \bigg( - x  + x_s e^{r(t-s)} +  \frac{y_s}{\theta} e^{\eta(t) \lambda(t) + \zeta(t)} e^{-r(T-t)}  \\
	&-  (\alpha(t) \lambda(t) + \beta(t) )e^{-r(T-t)}  + (\alpha(s) \lambda(s) + \beta(s) )e^{-r(T-t)} \bigg)  + \frac{1}{ k } \alpha(t)  e^{-r(T-t)},
	\end{cases}
	\end{align}
	and
	\begin{align}
	\label{mmvSN.strategy.optimal.b}
	\begin{cases}
	&o^*(t) = - \frac{ \mu_0 -  r }{  \sigma_0 }, \\
	&p^*(t,z)= \frac{\kappa_r \mu_1 z}{    \sigma_1^2}, \\
	&q^*(t,z)= e^{-\eta(t)z}+\phi(t)e^{-\eta(t)z}z-1.
	\end{cases}
	\end{align}
\end{theorem}

\subsection{MMV Efficient Frontier}

To simplify the notaitons, we introduce the following functions
\begin{align*}
\psi_1(s,t) :=& \frac{\kappa_r^2   \mu_1^2}{\delta \sigma_1^2}(e^{-\delta(t-s)}-e^{-\delta(T-s)}), \\
\psi_2(s,t) :=& \zeta(t) + \rho \int_s^t \int_{\mathbb{R}_{>0}}  e^{-\eta(u) z} (e^{\psi_1(u,t) z} - 1) (\phi(u) z + 1) F_2(dz) du, \\
\psi_3(s,t) :=&  \zeta(t) + \rho \int_s^t \int_{\mathbb{R}_{>0}}  e^{-\eta(u) z} (e^{\psi_1(u,t) z} - 1) (\phi(u) z + 1)^2 F_2(dz) du,
\end{align*}
and
\begin{align*}
C_1(s,t) :=& \frac{e^0_s e^3_{s,t} - (e^2_{s,t})^2}{(e^0_s-e^2_{s,t})^2}, \\
C_2(s,t) :=& \frac{e^0_s(e^3_{s,t}-e^2_{s,t})}{e^0_se^3_{s,t}-(e^2_{s,t})^2}  \frac{\alpha(t) \rho \iota_r \mu_2}{\delta} (1 - e^{-\delta(t-s)}) e^{-r(T-t)}, \\
C_3(s,t) :=& \frac{\alpha(t)^2 \rho \sigma_2^2}{2 \delta} (1 - e^{-2\delta(t-s)}) e^{-2r(T-t)} - \frac{(e^2_{s,t})^2}{e^0_s e^3_{s,t} - (e^2_{s,t})^2}  \frac{\alpha(t)^2 \rho^2 \iota_r^2 \mu_2^2}{\delta^2} (1 - e^{-\delta(t-s)})^2 e^{-2r(T-t)},
\end{align*}
where
$e^0_s := e^{\eta(s) \lambda + \zeta(s)}$, $e^2_{s,t} := e^{\psi_1(s,t) \lambda + \psi_2(s,t)}$, $e^3_{s,t} := e^{\psi_1(s,t) \lambda + \psi_3(s,t)}$. Now we state the efficient frontier for the MMV criterion problem, and show the results for some special cases.

\begin{theorem} \label{mmfSN.theorem.efficient_frontier}
	\par {\bf (Efficient Frontier)} For initial values $X^*(s) = x$, $Y^*(s) = y$ and $\lambda(s) = \lambda$, letting $\mathbb{V}ar_{s,x,y,\lambda}^P[\cdot]$ represents $\mathbb{V}ar^P[\cdot|X^a(s) = x,Y^{b}(s) = y,\lambda(s) = \lambda]$, then the expected wealth process under $Q^*$ is
	\begin{align*}
	\mathbb{E}_{s,x,y,\lambda}^{Q^*} X^*(t) =& x e^{r(t-s)}  - (\alpha(t) e^{-\delta(t-s)} - \alpha(s)) e^{-r(T-t)}  \lambda \\
	&- \alpha(t) \frac{\rho (\iota_r + 1 ) \mu_2}{\delta} (1 - e^{-\delta(t-s)}) e^{-r(T-t)} - \beta(t) e^{-r(T-t)} + \beta(s) e^{-r(T-t)}.
	\end{align*}
	Furthermore, the variance and expectation of the wealth process at time $t$ have the relationship:
	\begin{align}
	\label{mmvSN.equation.relationof.VarXandEX} \mathbb{V}ar_{s,x,y,\lambda}^PX^*(t) = C_1(s,t)  \bigg( \mathbb{E}_{s,x,y,\lambda}^P X^*(t) - \mathbb{E}_{s,x,y,\lambda}^{Q^*} X^*(t) - C_2(s,t) \bigg)^2 +  C_3(s,t).
	\end{align}
\end{theorem}

\begin{corollary}
	The efficient frontier for the terminal wealth process is given by
	\begin{align*}
	\mathbb{V}ar_{s,x,y,\lambda}^PX^*(T) = \frac{1}{e^{\eta(s) \lambda + \zeta(s)}-1}  \bigg( \mathbb{E}_{s,x,y,\lambda}^P X^*(T) - \mathbb{E}_{s,x,y,\lambda}^{Q^*} X^*(T)  \bigg)^2.
	\end{align*}
\end{corollary}
\begin{proof}
	\par Let $t = T$ in \eqref{mmvSN.equation.relationof.VarXandEX}.
\end{proof}

\begin{corollary}
	\par If there is no catastrophe in the model, i.e., $\lambda(t)$ is a constant $\lambda$ and $\delta = \rho = 0$, then
	\begin{align*}
	\mathbb{V}ar_{s,x,y,\lambda}^PX^*(t) = \frac{1 }{e^{\big(\frac{\lambda \kappa_r^2   \mu_1^2}{\sigma_1^2} + \frac{(\mu_0-r)^2}{\sigma_0^2}\big)(t-s) }-1 }  \bigg( \mathbb{E}_{s,x,y,\lambda}^P X^*(t) - \mathbb{E}_{s,x,y,\lambda}^{Q^*} X^*(t) \bigg)^2.
	\end{align*}
\end{corollary}
\begin{proof}
	\par Let $\delta = \rho = 0$ in \eqref{mmvSN.equation.relationof.VarXandEX}, we have
	\begin{align*}
	&\eta(s) = \frac{\kappa_r^2   \mu_1^2}{\sigma_1^2}(T-s),  \quad  \psi_1(s,t) = \frac{\kappa_r^2   \mu_1^2}{\sigma_1^2}(T-t), \\
	&\zeta(s) = \frac{(\mu_0-r)^2}{\sigma_0^2}(T-s), \quad \psi_2(s,t) = \psi_3(s,t) = \frac{(\mu_0-r)^2}{\sigma_0^2}(T-t), \\
	&e^0_s = e^{\big(\frac{\lambda \kappa_r^2   \mu_1^2}{\sigma_1^2} + \frac{(\mu_0-r)^2}{\sigma_0^2}\big)(T-s) }, \quad e^2_{s,t} = e^3_{s,t} = e^{\big(\frac{\lambda \kappa_r^2   \mu_1^2}{\sigma_1^2} + \frac{(\mu_0-r)^2}{\sigma_0^2}\big)(T-t) }.
	\end{align*}	
	Therefore,
	\begin{align*}
	C_1(s,t) = \frac{1 }{e^{\big(\frac{\lambda \kappa_r^2   \mu_1^2}{\sigma_1^2} + \frac{(\mu_0-r)^2}{\sigma_0^2}\big)(t-s) }-1 }, \quad C_2(s,t) = C_3(s,t) = 0.
	\end{align*}

\end{proof}

\section{Optimization Problem for the Diffusion Approximation Model} \label{mmvSN.section.approximation}

In this section, we consider the optimization problem under MMV for the diffusion approximation model. We take a parallel procedure as that of Section 2 and Section 3. In Subsection 4.1, we introduce the diffusion approximation model for the catastrophe insurance. In Subsection 4.2, we present the auxiliary problem, and by applying the dynamic programming principle, we gives the corresponding HJB equation. In Subsection 4.3, the optimal strategies and value function are obtained explicitly by solving the HJBI equation.

\subsection{Diffusion Approximation of the Catastrophe Insurance}

In \cite{dassios2005kalman-bucy}, the diffusion approximation of the Cox-process driven by shot noise intensity is obtained. By Theorem 2 of \cite{dassios2005kalman-bucy}, the diffusion approximation of the aggregate claims process and the intensity process are given by
\begin{align*}
dC_1(t) = \mu_1 \lambda(t) dt + \sigma_1\sqrt{\frac{ \rho \mu_2}{\delta}} dW_1(t),
\end{align*}
\begin{align*}
d C_2(t) =  k \rho \mu_2 dt + k \sigma_2\sqrt{\rho} dW_2(t),
\end{align*}
and
\begin{align*}
d \lambda(t) = (-\delta \lambda(t) + \rho \mu_2) dt + \sigma_2\sqrt{ \rho} dW_2(t).
\end{align*}
where $\mu_1$, $\sigma_1^2$, $\mu_2$, and $\sigma_2^2$ are defined by \eqref{mmvSN.moment}, $W_1(t)$ and $W_2(t)$ are two independent standard Brownian motions under probability $P$.
\par In this section, we consider the MMV optimization problem under this diffusion approximation. Similarly, the surplus process is given by
\begin{align*}
dR(t) =& \bigg(\kappa_r u(t) - (\kappa_r - \kappa)\bigg)\mu_1 \lambda(t) dt -  u(t) \sigma_1\sqrt{\frac{ \rho \mu_2}{\delta}} dW_1(t) \\
&+ \bigg(\iota_r v(t) - (\iota_r - \iota)\bigg) k \mu_2 \rho dt -  k v(t) \sigma_2\sqrt{ \rho} dW_2(t).
\end{align*}
Then the insurer's surplus process after investment is given by
\begin{align*}
dX(t) =  \bigg( r X(t)+\pi(t) (\mu_0 - r)  + (\kappa_r u(t) - \kappa_r + \kappa)\mu_1\lambda(t)  + (\iota_r v(t) - \iota_r + \iota) k \mu_2\rho\bigg) dt \\
+ \pi(t)  \sigma_0  dW_0(t)  -  u(t) \sigma_1\sqrt{\frac{ \rho \mu_2}{\delta}} dW_1(t) -  k v(t) \sigma_2\sqrt{ \rho} dW_2(t).
\end{align*}

\subsection{Hamilton-Jacobi-Bellman Equation}


For the diffusion approximation model, the corresponding characteristic process $Y(t)$ is given by
\begin{equation*}
\label{mmvSN.diffusion.Q.representation}
Y(t) = y + \int_s^t  Y(r) o(r) dW_0(r) + \int_s^t  Y(r) p(r) dW_1(r) + \int_s^t \ Y(r) q(r) dW_2(r), \quad \forall t \in [0,T].
\end{equation*}
The insurer aims to find the optimal strategies $a^* = (\pi^*,u^*,v^*)$ and $b^* = (o^*,p^*,q^*)$ such that for any admissible $(a,b)$ the objective function $J^{a,b}(s,x,y,\lambda)$ is maximized by $a^*$ and is minimized by $b^*$, where
\begin{equation}
\label{mmvSN.diffusion.obj}
J^{a,b}(s,x,y,\lambda) = \mathbb{E}_{s,x,y,\lambda}^P \bigg[X^{a}(T)Y^{b}(T)+\frac{1}{2\theta}(Y^{b}(T))^2 \bigg].
\end{equation}
Recall that the processes $X(t)$, $Y(t)$ and $\lambda(t)$ satisfy the following SDE:

\begin{subequations}
	\begin{empheq} [left=\empheqlbrace] {align}
	\nonumber dX(t) =&  \bigg( r X(t)+\pi(t) (\mu_0 - r)  + (\kappa_r u(t) - \kappa_r + \kappa)\mu_1\lambda(t)  + (\iota_r v(t) - \iota_r + \iota) k \mu_2\rho\bigg) dt \\
	\label{mmvSN.diffusion.dynamic.x} &+ \pi(t)  \sigma_0  dW_0(t)  -  u(t) \sigma_1\sqrt{\frac{ \rho \mu_2}{\delta}} dW_1(t) -  k v(t) \sigma_2\sqrt{ \rho} dW_2(t), \\
	\label{mmvSN.diffusion.dynamic.y} dY(t) =& Y(t) o(t) dW_0(t) + Y(t) p(t) dW_1(t) +  Y(t) q(t) dW_2(t), \\
	\label{mmvSN.diffusion.dynamic.lambda} d\lambda(t) =& (-\delta \lambda(t) + \rho \mu_2 ) dt + \sigma_2\sqrt{ \rho} dW_2(t).
	\end{empheq}
\end{subequations}

\begin{definition} \label{mmvSN.diffusion.definition.admissible}
	\par {\bf (Admissible Strategies)} The strategy $\{a(t)\}_{0\leq t\leq T} = \{(\pi(t),u(t),v(t))\}_{0\leq t\leq T}$ for player one is admissible in Problem \eqref{obj.3} if $\pi: [s,T] \to \mathbb{R}$, $u: [s,T] \to \mathbb{R}$, and $v: [s,T] \to \mathbb\mathbb{R}$ are $\mathbb{F}$-adapted processes that ensure the well-definedness of \eqref{mmvSN.diffusion.dynamic.x} and satisfy $\mathbb{E}^PX^u(t) < \infty$. We denote the set of all such admissible strategies as $\mathcal{A}[s,T]$.
	\par The strategy $\{b(t)\}_{0\leq t\leq T} = \{(o(t),p(t),q(t))\}_{0\leq t\leq T}$ for player two is admissible in Problem \eqref{obj.3} if $o: [s,T] \to \mathbb{R}$, $p: [s,T] \to \mathbb{R}$, and $q: [s,T] \to \mathbb{R}$ are $\mathbb{F}$-adapted processes such that the SDE \eqref{mmvSN.diffusion.dynamic.y} admits a unique solution, which is a nonnegative $\mathbb{F}$-adapted square integrable $P$-martingale, and it satisfies $\mathbb{E}^PY(t) = 1$ for $t \in [s,T]$. We denote the set of all such admissible strategies as $\mathcal{B}[s,T]$.
\end{definition}

We present the following verification theorem for the diffusion approximation case, the proof of which is similar to Theorem \ref{mmvSN.theorem.verification.0}.

\begin{theorem}
	{\bf (Verification Theorem)} Suppose that $W: (s,x,y,\lambda) \mapsto[s,T] \times \mathbb{R} \times \mathbb{R}_{\geq 0} \times \mathbb{R}_{>0}$ is a $C^{1,2,2,2}$ function satisfying the following condition
	\begin{align}
	\label{mmvSN.diffusion.hjbi.equation}
	\begin{cases}
	&\mathcal{L}_t^{a^*,b^*}W(t,x,y,\lambda) = 0, \qquad \forall (t,x,y,\lambda) \in [s,T) \times \mathbb{R} \times \mathbb{R}_{\geq 0} \times \mathbb{R}_{>0}, \\
	&\mathcal{L}_t^{a^*,b}W(t,x,y,\lambda) \geq 0, \qquad \forall b \in \mathcal{B}[s,T], \qquad \forall (t,x,y,\lambda) \in [s,T) \times \mathbb{R} \times \mathbb{R}_{\geq 0} \times \mathbb{R}_{>0}, \\
	&\mathcal{L}_t^{a,b^*}W(t,x,y,\lambda) \leq 0, \qquad \forall a \in \mathcal{A}[s,T] , \qquad \forall (t,x,y,\lambda) \in [s,T) \times \mathbb{R} \times \mathbb{R}_{\geq 0} \times \mathbb{R}_{>0}, \\
	&W(T,x,y,\lambda) = xy + \frac{1}{2 \theta} y^2,
	\end{cases}
	\end{align}
	where $\mathcal{L}_t^{a,b}$ is the infinitesimal generator of $(X(t),Y(t),\lambda(t))$ given by
	\begin{align}
	\nonumber \mathcal{L}_t^{a,b} W(t,x,y,\lambda) =& W_t +\biggl\{ r x+\pi (\mu_0 - r)  + (\kappa_r u - \kappa_r + \kappa)\mu_1\lambda  + (\iota_r v - \iota_r + \iota) k \mu_2\rho \biggr\} W_x \\
	\nonumber &+(-\delta \lambda + \rho \mu_2) W_{\lambda} + ( \frac{1}{2} \pi^2  \sigma_0^2 + \frac{1}{2} u^2 \sigma_1^2 \frac{ \rho \mu_2}{\delta} +  \frac{1}{2} k^2 v^2 \rho \sigma_2^2 ) W_{xx} \\
	\label{mmvSN.diffusion.generator} &+ \frac{1}{2} \rho \sigma_2^2 W_{\lambda \lambda} + (\frac{1}{2} y^2 o^2 + \frac{1}{2} y^2 p^2 + \frac{1}{2} y^2 q^2) W_{yy} \\
	\nonumber &+ (\pi \sigma_0 y o - u \sigma_1\sqrt{\frac{ \rho \mu_2}{\delta}} y p -  k v \sigma_2\sqrt{ \rho} y q ) W_{xy} - k v \rho \sigma_2^2 W_{x \lambda} + \sigma_2\sqrt{ \rho} y q W_{y \lambda}.
	\end{align}
	Then, 
	\begin{align*}
	W(s,x,y,\lambda) = \sup_{a \in \mathcal{A}[s,T]} \inf_{b \in \mathcal{B}[s,T]} J^{a,b}(s,x,y,\lambda).
	\end{align*}
\end{theorem}

\subsection{Solutions to the Diffusion Approximation Model}

We take a similar procedure as that of Section \ref{mmvSN.section.main_result} in this subsection. The explicit form of value functions and optimal strategies for the diffusion approximation model are also obtained. All the proofs in this section are delegated in Appendix B.

\begin{theorem} \label{mmvSN.diffusion.main_result}
	\par Assume the initial values $X(0) = x_0$, $Y(0) = 1$ and $\lambda(0) = \lambda_0$. The value function for Problem \eqref{obj.0} is given by
	\begin{align}
	W(t,x) = -\frac{\theta(e^{r(T-t)}x+\alpha(t)\lambda(t) + \beta(t))^2}{4e^{\xi(t)\lambda(t)^2 + \eta(t) \lambda(t) + \zeta(t)}} + \frac{\theta(e^{rT}x_0+\frac{2}{\theta}e^{\xi(0)\lambda_0^2 + \eta(0) \lambda_0 + \zeta(0)}+\alpha(0)\lambda_0 + \beta(0))^2}{4e^{\xi(t)\lambda(t)^2 + \eta(t) \lambda(t) + \zeta(t)}},
	\end{align}
	and the optimal feedback strategies are given as follows
	\begin{align*}
	\begin{cases}
	\pi^*(t,x) =& \frac{ \mu_0 -  r  }{ \sigma_0 ^2} \bigg( - x  + x_0 e^{rt} +  \frac{1}{\theta} e^{\xi(t)\lambda(t)^2 + \eta(t) \lambda(t) + \zeta(t)} e^{-r(T-t)}  \\
	&-  (\alpha(t) \lambda(t) + \beta(t) )e^{-r(T-t)}  + (\alpha(0) \lambda_0 + \beta(0) )e^{-r(T-t)} \bigg), \\
	u^*(t,x)=& \frac{\kappa_r \mu_1 }{ \sigma_1^2}  \bigg( - x  + x_0 e^{rt} +  \frac{1}{\theta} e^{\xi(t)\lambda(t)^2 + \eta(t) \lambda(t) + \zeta(t)} e^{-r(T-t)}  \\
	&-  (\alpha(t) \lambda(t) + \beta(t) )e^{-r(T-t)}  + (\alpha(0) \lambda_0 + \beta(0) )e^{-r(T-t)} \bigg), \\
	v^*(t,x) =& \frac{1}{ k } (\frac{\iota_r \mu_2}{\sigma_2^2} + 2\xi(t)\lambda(t)+\eta(t))  \bigg( - x  + x_0 e^{rt} +  \frac{1}{\theta} e^{\xi(t)\lambda(t)^2 + \eta(t) \lambda(t) + \zeta(t)} e^{-r(T-t)}  \\
	&-  (\alpha(t) \lambda(t) + \beta(t) )e^{-r(T-t)}  + (\alpha(0) \lambda_0 + \beta(0) )e^{-r(T-t)} \bigg)  + \frac{\alpha(t)}{k}   e^{-r(T-t)},
	\end{cases}
	\end{align*}	
	where
	\begin{align*}
	\begin{cases}
	\Delta =& 4 \delta^2 -   \frac{8 \kappa_r^2  \mu_1^2 \sigma_2^2 \delta}{\mu_2 \sigma_1^2}, \\
	d_{1,2} =& \frac{2 \delta \pm \sqrt{\Delta}}{4 \rho \sigma_2^2}, \\
	\xi(t) =& \frac{\kappa_r^2 \mu_1^2}{\sigma_1^2} \frac{T-t}{\rho \mu_2(T-t) + \frac{\rho \mu_2}{\delta}} 1_{\{\Delta = 0\}} + \frac{d_1d_2(e^{\sqrt{\Delta}(T-t)}-1)}{d_1e^{\sqrt{\Delta}(T-t)}-d_2} 1_{\{\Delta >0\}}, \\
	\eta(t) =& \frac{ ( 2 \iota_r + 1)\delta \kappa_r^2 \mu_1^2}{\sigma_1^2} \frac{(T-t)^2 }{\delta(T-t)+1} 1_{\{\Delta = 0\}} + \frac{4 \rho \mu_2 ( 2 \iota_r + 1) d_1d_2}{3 \sqrt{\Delta}}\frac{  \left ( e^{\frac{\sqrt{\Delta}}{2}(T-t)} -1\right )^2 \left ( 1 +2 e^{-\frac{\sqrt{\Delta}}{2}(T-t)} \right )}{d_1e^{\sqrt{\Delta}(T-t)}-d_2} 1_{\{\Delta >0\}}, \\
	\zeta(t) =& \rho \int_t^T \bigg(\rho \mu_2 ( 2 \iota_r + 1) \eta(s) + \frac{1}{2} \rho \sigma_2^2 \eta(s)^2 + \rho \sigma_2^2  \xi(s) \bigg) ds + \frac{ \rho (\mu_0 - r)^2 }{    \sigma_0^2  }(T-t)  + \frac{ \rho^2 \iota_r^2 \mu_2^2 }{   \sigma_2^2 } (T-t), \\
	\alpha(t) =& \frac{(\kappa_r - \kappa)\mu_1}{\delta + r} (e^{- \delta(T-t)} - e^{r(T-t)}), \\
	\beta(t) =& \frac{\rho (\kappa_r - \kappa) (\iota_r+1) \mu_1 \mu_2}{\delta + r} \bigg( \frac{1}{\delta}(1-e^{-\delta(T-t)})+\frac{1}{r}(1-e^{r(T-t)}) \bigg)  - \frac{\rho ( \iota_r - \iota)\mu_2}{r} (1-e^{r(T-t)}).
	\end{cases}
	\end{align*}
\end{theorem}

To prove Theorem \ref{mmvSN.diffusion.main_result}, we use a similar approach as in Section \ref{mmvSN.section.main_result}. We first find a candidate solution to the HJBI equation \eqref{mmvSN.diffusion.hjbi.equation}, then by Lemma \ref{mmvSN.difussion.lemma.relationship_of_XY} below, we provide expressions of value function and optimal strategies independent of process $Y^*(t)$, only as functions of $X(t)$ and $\lambda(t)$. Finally, by setting $s=0$, $Y(0)=\mathbb{E}[\frac{dQ}{dP}|\mathcal{F}_0]=\mathbb{E}[\frac{dQ}{dP}]=1$ and $\lambda(0)=\lambda_0$, we obtain the solutions to the original problem.

\begin{proposition} \label{mmvSN.diffusion.lemma.solution.suppose}
	\par Consider the following partial differential equations
	\begin{align}
	\label{mmvSN.diffusion.equation.function.value.coefficients}
	\begin{cases}
	0 =& G_t +  r G(t), \\
	0 =& H_t + (-\delta \lambda +  \rho (2 \iota_r+1) \mu_2) H_{\lambda} + \frac{1}{2} \rho \sigma_2^2 H_{\lambda \lambda} + H(t,\lambda) \bigg(\frac{ (\mu_0 - r)^2 }{    \sigma_0^2  } + \frac{ \kappa_r^2 \mu_1^2 \lambda^2 \delta }{   \rho \mu_2 \sigma_1^2  } + \frac{ \rho \iota_r^2 \mu_2^2 }{   \sigma_2^2 }\bigg), \\
	0 =& I_t + \bigg((- \kappa_r + \kappa)\mu_1\lambda  + ( - \iota_r + \iota) k \rho \mu_2\bigg)G(t)   + (-\delta \lambda  +  \rho (\iota_r+1) \mu_2 )  I_{\lambda} + \frac{1}{2} \rho \sigma_2^2  I_{\lambda \lambda},
	\end{cases}
	\end{align}
	with boundry conditions $G(T) = 1$, $H(T,\lambda) = \frac{1}{2\theta}$, $I(T,\lambda) = 0$. 
	If there exist smooth enough solutions $G(t)$, $H(t,\lambda)$ and $I(t,\lambda)$ for the above PDEs such that, for all $t \in [0,T]$ and $\lambda \in \mathbb{R}_{>0}$, $H(t,\lambda) > 0$ holds, then the solution to the HJB equation \eqref{mmvSN.diffusion.hjbi.equation} is of the following form:
	\begin{equation}
	\label{mmvSN.diffusion.value.function.W}
	W(t,x,y,\lambda) = G(t) xy + H(t,\lambda)y^2 +I(t,\lambda)y.
	\end{equation}
	Moreover, the corresponding optimal feedback strategies is given by
	\begin{numcases}{}
	\nonumber
	\pi^*(t,x,y,\lambda)= \frac{2 H(t,\lambda)(\mu_0 - r)  y}{   G(t) \sigma_0^2  },  \\
	\nonumber
	u^*(t,x,y,\lambda)= \frac{2 H(t,\lambda) \kappa_r \mu_1 \lambda  y}{ G(t)  \sigma_1^2\frac{ \rho \mu_2}{\delta}  },  \\
	\nonumber
	v^*(t,x,y,\lambda)= \frac{1}{k}\bigg(\frac{2 H(t,\lambda) \iota_r \mu_2    y }{   G(t) \sigma_2^2 } + \frac{2 H_{\lambda} y }{   G(t) } + \frac{I_{\lambda}}{G(t)}\bigg),  \\
	\label{mmvSN.difussion.strategy.saddlepoint.equation}
	o^*(t,x,y,\lambda)=   -  \frac{\mu_0 - r}{   \sigma_0  },  \\
	\nonumber
	p^*(t,x,y,\lambda)=   \frac{2 \kappa_r \mu_1 \lambda  }{   \sigma_1\sqrt{\frac{ \rho \mu_2}{\delta}}  },  \\
	\nonumber
	q^*(t,x,y,\lambda)= \frac{\iota_r \mu_2 \sqrt{ \rho}   }{  \sigma_2 }.
	\end{numcases}
\end{proposition}

The optimal strategies \eqref{mmvSN.difussion.strategy.saddlepoint.equation} depend on the auxiliary process $Y(t)$, which is not able to be observed in practice. Now we give an analogue of Lemma \ref{mmvSN.lemma.relationship_of_XY} to show the relationship between $X(t)$ and $Y(t)$.

\begin{lemma} \label{mmvSN.difussion.lemma.relationship_of_XY}
	\par Let $X^*(t)$ and $Y^*(t)$ be the state processes under the feedback strategies defined by \eqref{mmvSN.difussion.strategy.saddlepoint.equation}, then
	\begin{align}
	\label{mmvSN.difussion.equation.relationship_of_XY} X^*(t) G(t) - X^*(s) G(s)= - 2   Y^*(t) H(t,\lambda(t)) + 2 Y^*(s) H(s,\lambda(s))  -  I(t,\lambda(t)) + I(s,\lambda(s)).
	\end{align}
\end{lemma} 
\begin{proof}
	\par By the same method of Lemma \ref{mmvSN.lemma.relationship_of_XY}, it is easy to prove that
	\begin{align*}
	d\biggl(X^*(t) G(t)\biggr) = - 2 d \bigg( Y^*(t) H(t,\lambda(t)) \bigg) - d I(t,\lambda(t)).
	\end{align*}	
\end{proof}

Note that the process $\lambda(t)$ does not depend on the controls, by substituting \eqref{mmvSN.difussion.equation.relationship_of_XY} into \eqref{mmvSN.diffusion.value.function.W} and  \eqref{mmvSN.difussion.strategy.saddlepoint.equation}, we can derive the optimal feedback strategies as follows:
\begin{theorem}
	\par  Let $X^*(t)$ and $Y^*(t)$ be the state processes under the feedback strategies defined by \eqref{mmvSN.strategy.saddlepoint.equation}, then for any $0 \leq s \leq t \leq T$, the value function is given by
	\begin{align}
	W(t,X^*(t)) = -\frac{(G(t)X^*(t)+I(t,\lambda(t)))^2}{4H(t,\lambda(t))} + \frac{(G(s)X^*(s)+2H(s,\lambda(s)Y^*(s)+I(s,\lambda(s)))^2}{4H(t,\lambda(t))}+K(t,\lambda(t)),
	\end{align}
	and the optimal strategies $\pi^*$, $u^*$ and $v^*$ at time $t$ satisfy
	\begin{align*}
	\begin{cases}
	\pi^*(t) =& \frac{ \mu_0 -  r  }{ \sigma_0 ^2} \bigg( - X^*(t)  + X^*(s) \frac{G(s)}{G(t)} + 2 Y^*(s) \frac{H(s,\lambda(s))}{G(t)}  -  \frac{I(t,\lambda(t))}{G(t)} + \frac{I(s,\lambda(s))}{G(t)} \bigg), \\
	u^*(t) =& \frac{ \kappa_r \mu_1 \lambda(t) }{  \sigma_1^2\frac{ \rho \mu_2}{\delta}  }  \bigg( - X^*(t)  + X^*(s) \frac{G(s)}{G(t)} + 2 Y^*(s) \frac{H(s,\lambda(s))}{G(t)}  -  \frac{I(t,\lambda(t))}{G(t)} + \frac{I(s,\lambda(s))}{G(t)} \bigg), \\
	v^*(t) =& \frac{1}{k}\bigg(\frac{\iota_r \mu_2 }{ \sigma_2^2 } + \frac{ H_{\lambda}(t,\lambda(t)) }{ H(t,\lambda(t)) }\bigg) \bigg( - X^*(t)  + X^*(s) \frac{G(s)}{G(t)} + 2 Y^*(s) \frac{H(s,\lambda(s))}{G(t)}  -  \frac{I(t,\lambda(t))}{G(t)} + \frac{I(s,\lambda(s))}{G(t)} \bigg) \\
	&+ \frac{I_{\lambda}(t,\lambda(t))}{kG(t)}.
	\end{cases}
	\end{align*}
\end{theorem}

Next, we give the explicit forms of the solutions to \eqref{mmvSN.diffusion.equation.function.value.coefficients}.
\begin{lemma} \label{mmvSN.diffusion.lemma.coefficients}
	\par If $ \frac{\delta \mu_2}{\sigma_2^2} \geq \frac{2 \kappa_r^2  \mu_1^2 }{ \sigma_1^2}$, then the solutions to \eqref{mmvSN.diffusion.equation.function.value.coefficients} are as follows
	\begin{numcases}{}
	\nonumber
	G(t) = e^{r(T-t)}, \\
	\nonumber
	H(t,\lambda) = \frac{1}{2 \theta}e^{\xi(t) \lambda^2 + \eta(t) \lambda + \zeta(t)}, \\
	\label{mmvSN.diffusion.function.value.coefficients.solution}
	I(t,\lambda) = \alpha(t) \lambda +\beta(t),
	\end{numcases}
	where
	\begin{align*}
	\begin{cases}
	\Delta =& 4 \delta^2 -   \frac{8 \kappa_r^2  \mu_1^2 \sigma_2^2 \delta}{\mu_2 \sigma_1^2}, \\
	d_{1,2} =& \frac{2 \delta \pm \sqrt{\Delta}}{4 \rho \sigma_2^2}, \\
	\xi(t) =& \frac{\kappa_r^2 \mu_1^2}{\sigma_1^2} \frac{T-t}{\rho \mu_2(T-t) + \frac{\rho \mu_2}{\delta}} 1_{\{\Delta = 0\}} + \frac{d_1d_2(e^{\sqrt{\Delta}(T-t)}-1)}{d_1e^{\sqrt{\Delta}(T-t)}-d_2} 1_{\{\Delta >0\}}, \\
	\eta(t) =& \frac{ ( 2 \iota_r + 1)\delta \kappa_r^2 \mu_1^2}{\sigma_1^2} \frac{(T-t)^2 }{\delta(T-t)+1} 1_{\{\Delta = 0\}} + \frac{4 \rho \mu_2 ( 2 \iota_r + 1) d_1d_2}{3 \sqrt{\Delta}}\frac{  \left ( e^{\frac{\sqrt{\Delta}}{2}(T-t)} -1\right )^2 \left ( 1 +2 e^{-\frac{\sqrt{\Delta}}{2}(T-t)} \right )}{d_1e^{\sqrt{\Delta}(T-t)}-d_2} 1_{\{\Delta >0\}}, \\
	\zeta(t) =& \rho \int_t^T \bigg(\rho \mu_2 ( 2 \iota_r + 1) \eta(s) + \frac{1}{2} \rho \sigma_2^2 \eta(s)^2 + \rho \sigma_2^2  \xi(s) \bigg) ds + \frac{ \rho (\mu_0 - r)^2 }{    \sigma_0^2  }(T-t)  + \frac{ \rho^2 \iota_r^2 \mu_2^2 }{   \sigma_2^2 } (T-t), \\
	\alpha(t) =& \frac{(\kappa_r - \kappa)\mu_1}{\delta + r} (e^{- \delta(T-t)} - e^{r(T-t)}), \\
	\beta(t) =& \frac{\rho (\kappa_r - \kappa) (\iota_r+1) \mu_1 \mu_2}{\delta + r} \bigg( \frac{1}{\delta}(1-e^{-\delta(T-t)})+\frac{1}{r}(1-e^{r(T-t)}) \bigg)  - \frac{\rho ( \iota_r - \iota)\mu_2}{r} (1-e^{r(T-t)}).
	\end{cases}
	\end{align*}
\end{lemma}

\par By far, from the above analysis and computation, we are able to find the solution of HJBI equation \eqref{mmvSN.diffusion.hjbi.equation}:

\begin{theorem} \label{mmvSN.diffusion.theorem.solution.explicit}
	\par {\bf (solution of HJBI equation)} Assume the initial values $X(s) = x_s$ and $Y(s) = y_s$. The value function related to the diffusion approximation is given by $W(t,x) \in C^{1,2}([s,T] \times \mathbb{R})$ as follows,
	\begin{align}
	\label{mmvSN.diffusion.function.value} 	W(t,x;s,x_s,y_s) = -\frac{\theta(e^{r(T-t)}x+\alpha(t)\lambda(t) + \beta(t))^2}{4e^{\xi(t)\lambda(t)^2 + \eta(t) \lambda(t) + \zeta(t)}} + \frac{\theta(e^{rT}x_s+\frac{2 y_s}{\theta}e^{\xi(s)\lambda(s)^2 + \eta(s) \lambda(s) + \zeta(s)}+\alpha(s)\lambda(s) + \beta(s))^2}{4e^{\xi(t)\lambda(t)^2 + \eta(t) \lambda(t) + \zeta(t)}},
	\end{align}
	and the equilibrium strategies are precommitted feedback
	\begin{align}
	\label{mmvSN.diffusion.strategy.optimal.a} 
	\begin{cases}
	\pi^*(t,x;s,x_s,y_s) =& \frac{ \mu_0 -  r  }{ \sigma_0 ^2} \bigg( - x  + x_s e^{r(t-s)} +  \frac{y_s}{\theta} e^{\xi(t)\lambda(t)^2 + \eta(t) \lambda(t) + \zeta(t)} e^{-r(T-t)}  \\
	&-  (\alpha(t) \lambda(t) + \beta(t) )e^{-r(T-t)}  + (\alpha(s) \lambda(s) + \beta(s) )e^{-r(T-t)} \bigg), \\
	u^*(t,x;s,x_s,y_s)=& \frac{\kappa_r \mu_1 }{ \sigma_1^2}  \bigg( - x  + x_s e^{r(t-s)} +  \frac{y_s}{\theta} e^{\xi(t)\lambda(t)^2 + \eta(t) \lambda(t) + \zeta(t)} e^{-r(T-t)}  \\
	&-  (\alpha(t) \lambda(t) + \beta(t) )e^{-r(T-t)}  + (\alpha(s) \lambda(s) + \beta(s) )e^{-r(T-t)} \bigg), \\
	v^*(t,x;s,x_s,y_s) =& \frac{1}{ k } (\frac{\iota_r \mu_2}{\sigma_2^2} + 2\xi(t)\lambda(t)+\eta(t))  \bigg( - x  + x_s e^{r(t-s)} +  \frac{y_s}{\theta} e^{\xi(t)\lambda(t)^2 + \eta(t) \lambda(t) + \zeta(t)} e^{-r(T-t)}  \\
	&-  (\alpha(t) \lambda(t) + \beta(t) )e^{-r(T-t)}  + (\alpha(s) \lambda(s) + \beta(s) )e^{-r(T-t)} \bigg)  + \frac{1}{ k } \alpha(t)  e^{-r(T-t)},
	\end{cases}
	\end{align}
	and
	\begin{align}
	\label{mmvSN.diffusion.strategy.optimal.b}
	\begin{cases}
	&o^*(t) = - \frac{ \mu_0 -  r }{  \sigma_0 }, \\
	&p^*(t,z)= \frac{2 \kappa_r \mu_1 \lambda  }{   \sigma_1\sqrt{\frac{ \rho \mu_2}{\delta}}  }, \\
	&q^*(t,z)= \frac{\iota_r \mu_2 \sqrt{ \rho}   }{  \sigma_2 }.
	\end{cases}
	\end{align}
\end{theorem}

\iftrue
\section{Numerical Examples} \label{mmvSN.section.numerical}

In this section, we assume $U_i$ is exponentially distributed with parameter $\beta_1 = 0.2$ and $V_i$ is exponentially distributed with parameter $\beta_2 = 0.3$. Other parameters are supposed to be
\begin{itemize}
	\item $T = 100$ months is the time horizon of planning;
	\item $s = 0$ is the initial time;
	\item $\mu_0 = 0.03, \sigma_0 = 0.4$ are the monthly expected return rate and the expected volatility of the stock, respectively;
	\item $r = 0.01$ is the monthly return of the risk-free asset;
	\item $x_0 = 100, \quad \lambda_0 = 1$ are the wealth of the insurer and the intensity of the ordinary insurance claims at the beginning;
	\item $\theta = 1$ is the factor of risk aversion in the MMV criterion;
	\item $\kappa = 0.1, \quad \kappa_r = 0.105, \quad (\iota = 0.1, \quad \iota_r = 0.12)$ are the safety loading of the insurer and reinsurer for the ordinary insurance claims (catastrophe insurance claims) respectively;
	\item $\rho=0.01, \quad \delta=0.01$ are the intensity and the decay factor of the shot noise process. These two parameters mean the frequency of the occurrence of the catastrophe and the speed at which the effects of the catastrophe recede, respectively;
	\item $k = 10000$ is the ratio of the catastrophe claims size to the catastrophe impact.
\end{itemize}

We have generated two graphs to visualize the sensitivity of the optimal strategies $u$ and $v$, with respect to changes in catastrophe-related parameters. In Figs.\ref{mmvSN.fig:plot_u} and \ref{mmvSN.fig:plot_v}, the red lines depict the optimal strategies of ordinary claims $u$ and that of catastrophe claims $v$ as functions of the surplus $x$ under the given parameter values. Meanwhile, the other colored lines represent variations in the optimal strategies when one parameter is altered at a time.

In Fig.\ref{mmvSN.fig:plot_u}, observe the green and dark green lines, which illustrate that as the frequency and intensity of catastrophes increase (indicated by larger values of $\rho$ or smaller values of $\beta_2$), the insurer adopts a more aggressive stance, leading to higher retention levels. This behavior can be attributed to the fact that with increased catastrophe frequency and intensity, the catastrophe insurance premium $\frac{(1+\iota) \rho k }{\beta_2}$ also rises, allowing the insurer to shoulder more risk.

In the same figure, the orange line shows that when the impact of a catastrophe fades more fast (corresponding to a larger $\delta$), the insurer adopts a more conservative approach with smaller retention levels. This is because the insurer's ordinary insurance premium $\frac{(1+\kappa) \lambda(t)}{\beta_1}$ also diminishes more gradually for larger $\delta$.

Fig.\ref{mmvSN.fig:plot_v} illustrates how the optimal strategy $v$ of catastrophe claims changes in response to variations in the parameters. The orange line indicates that a larger ratio results in reduced retention. Furthermore, the purple lines in both figures highlight that higher reinsurance costs lead to greater retained risks. Finally, the blue lines illustrate that increased reinsurance premium incentivize the insurer to adopt a more aggressive strategy.

\begin{figure}[H]
	\begin{minipage}[t]{0.5\linewidth}
		\centering
		\includegraphics[width=\linewidth]{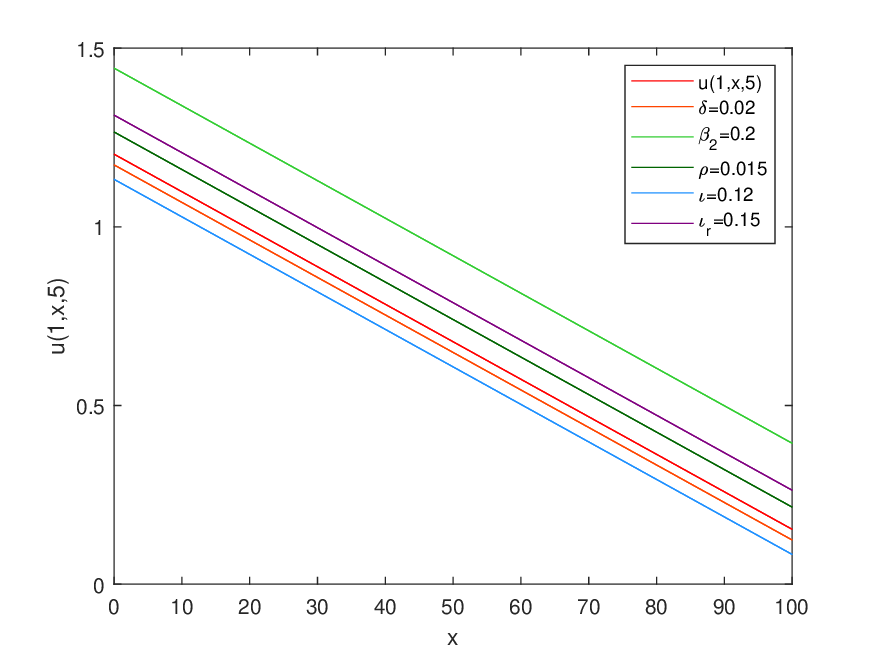}
	\end{minipage}
	\begin{minipage}[t]{0.5\linewidth}
		\centering
		\includegraphics[width=\linewidth]{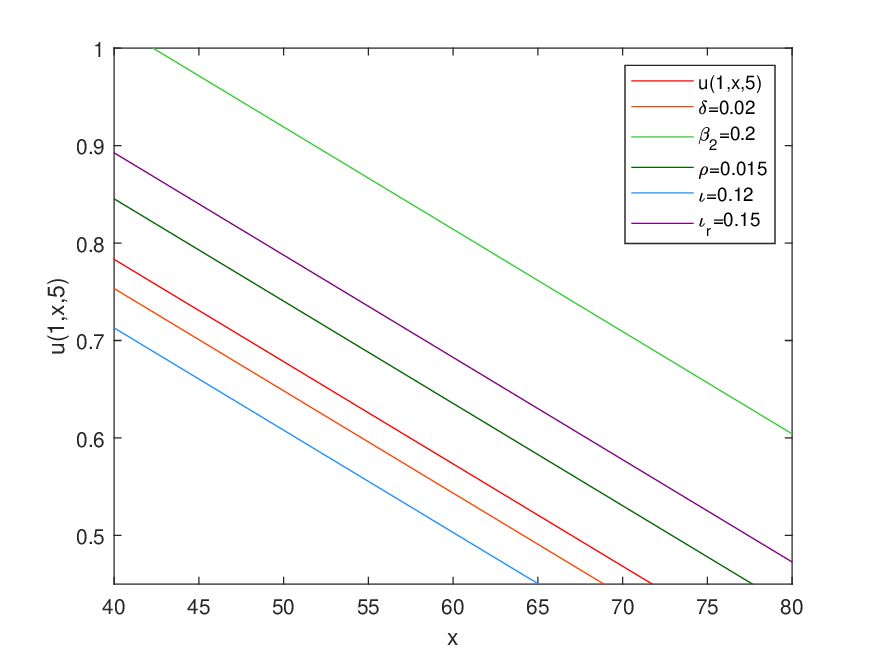}
	\end{minipage}
	\caption{The optimal retention level $u$ of the ordinary insurance when $t=1$ and $\lambda(1)=5$.}
	\label{mmvSN.fig:plot_u}
\end{figure}

\begin{figure}[H]
	\begin{minipage}[t]{0.5\linewidth}
		\centering
		\includegraphics[width=\linewidth]{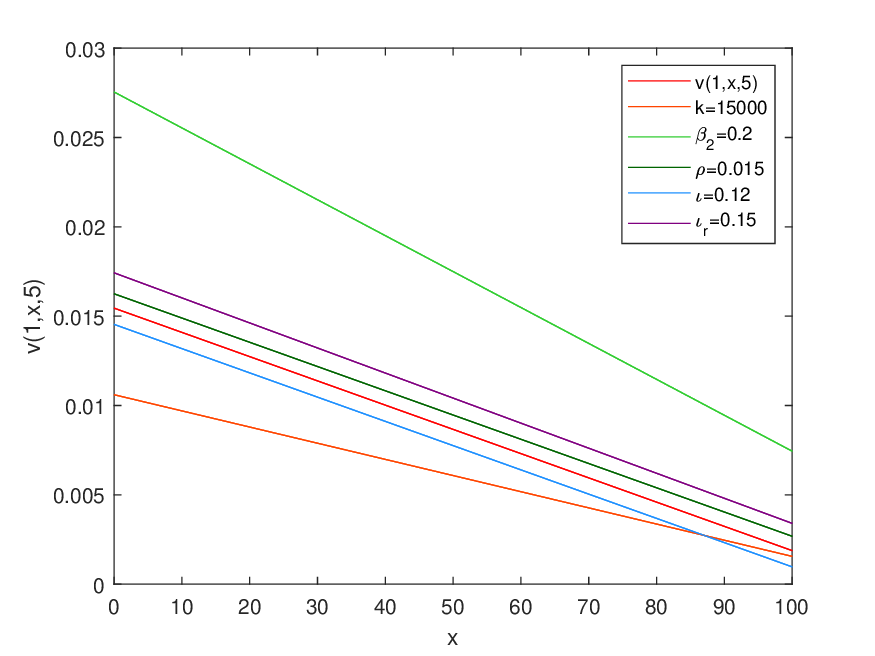}
	\end{minipage}
	\begin{minipage}[t]{0.5\linewidth}
		\centering
		\includegraphics[width=\linewidth]{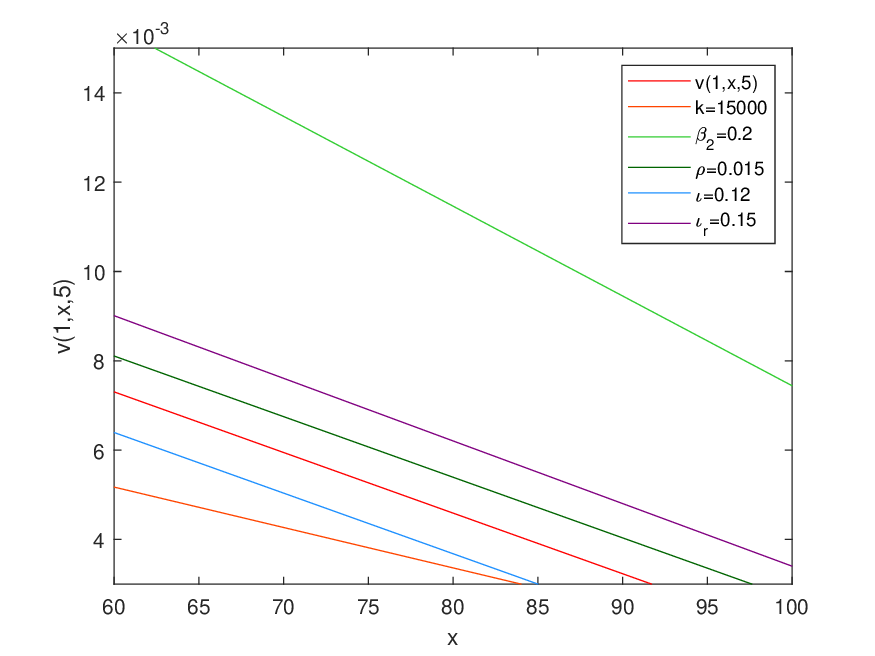}
	\end{minipage}
	\caption{The optimal retention level $v$ of the catastrophe insurance when $t=1$ and $\lambda(1)=5$.}
	\label{mmvSN.fig:plot_v}
\end{figure}

In Fig.\ref{mmvSN.fig:EF}, we plot the efficient frontiers for this problem at different time points, where the top half of the parabolas are the efficient frontiers, from which one can clearly see that when the risk (variance) is given, the expected return becomes larger over time.

\begin{figure}[H]
	\centering
	\includegraphics[width=\textwidth]{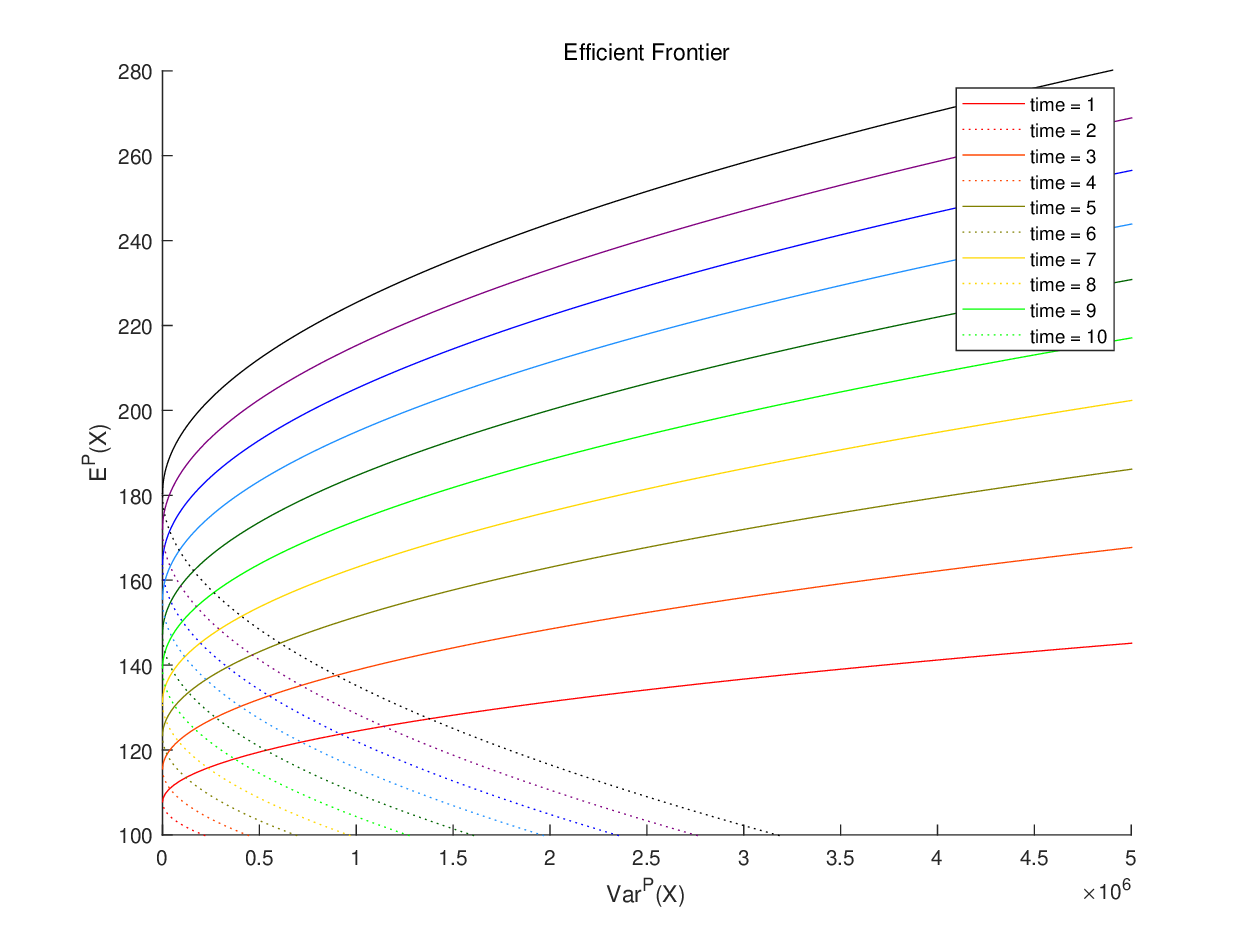}
	\caption{Efficient frontier for different time $t$.}
	\label{mmvSN.fig:EF}
\end{figure}
\fi

\section{Conclusion}

In this paper, we have studied the optimization problem of optimal investment and reinsurance for insurers facing catastrophic risks. Our study includes both jump models and diffusion approximation models, and the insurer's primary goal is to maximize their terminal surplus under the MMV criterion. We first formulate the initial control problem as an auxiliary two-player zero-sum game, then find equilibrium solutions in explicit form through dynamic programming and solving an HJBI equation. Furthermore, we have presented the efficient frontier within the MMV criterion. Looking ahead, our future research endeavors may include: 1) integrating excess-of-loss reinsurance into the catastrophe model; 2) investigating scenarios where the insurer has access to only partial information about catastrophes, limiting their ability to design tailored premiums; 3) exploring games involving multiple insurance companies or games between insurers and SPVs; and more.

~\\

\noindent {\bf Funding:} This work was supported by the National Natural Science Foundation of China 11931018 and 12271274.

\section*{Declarations}

\noindent {\bf Conflict of interest:} The authors have no relevant financial or non-financial interests to disclose.

\bibliographystyle{abbrv}
\bibliography{mono-MV}

\begin{appendices}
	
\section{Proofs of Statements in Section \ref{mmvSN.section.main_result}}

\subsection{Proof of Proposition \ref{mmvSN.lemma.solution.suppose}}
\begin{proof}
\par We suppose that the form of solution to \eqref{mmvSN.hjbi.equation} is given as follows
\begin{align}
\label{mmvSN.function.value.suppose}
W(t,x,y,\lambda) =  G(t)xy + H(t,\lambda)y^2 +I(t,\lambda)y + K(t,\lambda).
\end{align}
Substituting \eqref{mmvSN.function.value.suppose} into \eqref{mmvSN.generator} yields
\begin{align}
\nonumber \mathcal{L}_t^{a,b} W(t,x,y,\lambda) =& G_txy + H_t y^2 +I_t y + K_t + \biggl\{ r x + ( - \kappa_r + \kappa)\mu_1\lambda  \\
\nonumber &+ ( - \iota_r + \iota)k\mu_2\rho\biggr\} G(t) y  +(-\delta \lambda) ( H_{\lambda} y^2 +I_{\lambda} y + K_{\lambda} ) +  \rho \int_{\mathbb{R}_{>0}} \{  H_z y^2 + I_z y  + K_z \} F_2(dz)  \\
\nonumber &+ \pi( \mu_0 -  r ) G(t) y  + \pi  \sigma_0  y o G(t) + y^2 o^2 H(t,\lambda) \\
\nonumber &+ \kappa_ru \mu_1\lambda G(t) y + \lambda \int_{\mathbb{R}_{>0}} \{  - G(t)u zyp + H(t,\lambda)y^2 p^2 \} F_1(dz) \\
\label{mmvSN.generator.suppose} &+ \iota_r k v \mu_2\rho G(t) y + \rho \int_{\mathbb{R}_{>0}} \{ (2 H_z y + I_z -G(t)kvz)yq + H(t,\lambda+z)y^2q^2\} F_2(dz).
\end{align}

\par By differentiating \eqref{mmvSN.generator.suppose} with respect to $o$, $p$ and $q$ and letting them equal $0$, if $H(t,\lambda) > 0$ holds, the minimum of \eqref{mmvSN.generator.suppose} is attained at  
\begin{align*}
&o^* = - \frac{ G(t) \pi  \sigma_0  }{2H(t,\lambda)y}, \\
&p^*(z) = \frac{ G(t) u z}{2H(t,\lambda)y}, \\
&q^*(z) = -\frac{2 H_z y + I_z -G(t)kvz }{2H(t,\lambda+z)y}.
\end{align*}

\par We then plug $o^*$, $p^*(z)$ and $q^*(z)$ back into \eqref{mmvSN.generator.suppose} and note that 
\begin{align*}
&\rho \int_{\mathbb{R}_{>0}} \{ (2 H_z y + I_z -G(t)kvz)yq^*(z) + H(t,\lambda+z)y^2q^*(z)^2\} F_2(dz) \\
=& -2\rho  y^2 {\bf \widetilde{H}}(\frac{H_z^2}{z}) - \frac{\rho}{2} {\bf \widetilde{H}}(\frac{I_z^2}{z}) - \frac{\rho}{2} G(t)^2 k^2 v^2 {\bf \widetilde{H}}(z) - 2 \rho y {\bf \widetilde{H}}(\frac{H_zI_z}{z}) + 2 \rho G(t) y k v {\bf \widetilde{H}}(H_z) + \rho G(t) k v {\bf \widetilde{H}}(I_z).
\end{align*}
Therefore, \eqref{mmvSN.generator.suppose} becomes
\begin{align}
\nonumber &\mathcal{L}_t^{a,b^*} W(t,x,y,\lambda) = G_txy + H_t y^2 +I_t y + K_t + \biggl\{ r x + ( - \kappa_r + \kappa)\mu_1\lambda  \\
\nonumber &+ ( - \iota_r + \iota)k\mu_2\rho\biggr\} G(t) y  +(-\delta \lambda) ( H_{\lambda} y^2 +I_{\lambda} y + K_{\lambda} ) +  \rho \int_{\mathbb{R}_{>0}} \{  H_z y^2 + I_z y  + K_z \} F_2(dz)  \\
\nonumber &+ \pi( \mu_0 -  r ) G(t) y -  \frac{ \pi^2  \sigma_0 ^2 G(t)^2}{4H(t,\lambda)} + \kappa_ru \mu_1\lambda G(t) y - \lambda \frac{ G(t)^2 u^2 \sigma_1^2}{4H(t,\lambda)} \\
\label{mmvSN.generator.suppose.2} &+ \iota_r kv \mu_2\rho G(t) y  - \frac{\rho}{2} G(t)^2 k^2 v^2 {\bf \widetilde{H}}(z)  + 2 \rho G(t) y k v {\bf \widetilde{H}}(H_z) + \rho G(t) k v {\bf \widetilde{H}}(I_z) \\
\nonumber &-2\rho  y^2 {\bf \widetilde{H}}(\frac{H_z^2}{z}) - \frac{\rho}{2} {\bf \widetilde{H}}(\frac{I_z^2}{z}) - 2 \rho y {\bf \widetilde{H}}(\frac{H_zI_z}{z}),
\end{align}
where
\begin{align*}
{\bf \widetilde{H}}(f) := \int_{\mathbb{R}_{>0}} \frac{f(z) z}{2 H(t,\lambda+z)}F_2(dz).
\end{align*}

\par Similarly, by differentiating \eqref{mmvSN.generator.suppose.2} with respect to $\pi$, $u$ and $v$ and letting it equals $0$, if $G(t) > 0$ and $H(t,\lambda) > 0$ hold, the maximum of \eqref{mmvSN.generator.suppose.2} is attained at 
\begin{align*}
&\pi^* = \frac{2H(t,\lambda) ( \mu_0 -  r ) y }{ G(t)    \sigma_0 ^2}, \\
&u^* = \frac{2H(t,\lambda)\kappa_r \mu_1 y}{ G(t)   \sigma_1^2}, \\
&v^* = \frac{1}{kG(t)}\biggl\{\frac{\iota_r   \mu_2 y }{ {\bf \widetilde{H}}(z)} + \frac{2y{\bf \widetilde{H}}(H_z)}{ {\bf \widetilde{H}}(z)} + \frac{{\bf \widetilde{H}}(I_z)}{ {\bf \widetilde{H}}(z)}\biggr\}.
\end{align*}

\par By plugging $\pi^*$, $u^*$ and $v^*$ back into \eqref{mmvSN.generator.suppose.2}, we obtain
\begin{align*}
&\mathcal{L}_t^{a^*,b^*} W(t,x,y,\lambda) =   xy \biggl\{ G_t +  r G(t) \biggr\} \\
&+ y^2 \biggl\{ H_t -\delta \lambda H_{\lambda} +  \rho \int_{\mathbb{R}_{>0}} H_z F_2(dz) + \frac{H(t,\lambda) ( \mu_0 -  r )^2 }{    \sigma_0 ^2} + \frac{H(t,\lambda)  \lambda \kappa_r^2  \mu_1^2 }{ \sigma_1^2} \\
&+ \frac{\rho \iota_r^2  \mu_2^2 }{ {\bf \widetilde{H}}(z)} + \frac{2 {\bf \widetilde{H}}(H_z) \rho \iota_r  \mu_2 }{ {\bf \widetilde{H}}(z)} - \rho  \biggl\{2 {\bf \widetilde{H}}(\frac{H_z^2}{z}) - 2 \frac{{\bf \widetilde{H}}(H_z)^2}{{\bf \widetilde{H}}(z)} + \frac{\iota_r^2 \mu_2^2}{2{\bf \widetilde{H}}(z)}\biggr\}  \\
&+ y \biggl\{ I_t -\delta \lambda I_{\lambda} +  \rho \int_{\mathbb{R}_{>0}} I_z F_2(dz) + \bigg( ( - \kappa_r + \kappa)\mu_1\lambda  + ( - \iota_r + \iota) k\mu_2\rho\bigg) G(t) \\
&+ \frac{ {\bf \widetilde{H}}(I_z) \rho \iota_r  \mu_2 }{ {\bf \widetilde{H}}(z)} - \rho \biggl\{ 2 {\bf \widetilde{H}}(\frac{H_z I_z}{z}) - 2 \frac{{\bf \widetilde{H}}(H_z) {\bf \widetilde{H}}(I_z)}{{\bf \widetilde{H}}(z)} \biggr\} \\
&+\biggl\{ K_t -\delta \lambda K_{\lambda} +  \rho \int_{\mathbb{R}_{>0}} K_z F_2(dz) - \frac{\rho}{2} \biggl\{ {\bf \widetilde{H}}(\frac{I_z^2}{z}) - \frac{{\bf \widetilde{H}}(I_z)^2}{{\bf \widetilde{H}}(z)} \biggr\}   \biggr\}.
\end{align*}

\par By separation of variables, we get the partial differential equations \eqref{mmvSN.equation.function.value.coefficients} with the boundary conditions $G(T) = 1$, $H(T,\lambda) = \frac{1}{2\theta}$ and $I(T,\lambda) = K(T,\lambda) = 0$ which are derived from $W(T,x,y,\lambda) = xy + \frac{1}{2\theta}y^2$.

\end{proof}

\subsection{Proof of Lemma \ref{mmvSN.lemma.relationship_of_XY}}
\begin{proof}
	
	\par It is sufficient to prove that
	\begin{align*}
	d\biggl(X^*(t) G(t)\biggr) = - 2 d \bigg( Y^*(t) H(t,\lambda(t)) \bigg) - d I(t,\lambda(t)).
	\end{align*}

	\par Substituting \eqref{mmvSN.strategy.saddlepoint.equation} into \eqref{mmvSN.dynamic.x} and \eqref{mmvSN.dynamic.y} gives
	\begin{align*}
	dX^*(t) =&  \bigg( r X(t)- (\kappa_r - \kappa)\mu_1\lambda(t)   - ( \iota_r - \iota)k\mu_2\rho\bigg) dt \\
	&+  \frac{2H(t,\lambda(t)) ( \mu_0 -  r )^2 Y^*(t) }{ G(t)    \sigma_0 ^2}  dt +  \lambda(t) \frac{2H(t,\lambda(t))\kappa_r^2 \mu_1^2 Y^*(t)}{ G(t)   \sigma_1^2} dt  \\
	&+ \iota_r \mu_2\rho \frac{1}{G(t)}\biggl\{\frac{\iota_r  \mu_2 Y^*(t) }{ {\bf \widetilde{H}}(z)} + \frac{2 Y^*(t) {\bf \widetilde{H}}(H_z) }{ {\bf \widetilde{H}}(z)} + \frac{{\bf \widetilde{H}}(I_z)}{ {\bf \widetilde{H}}(z)}\biggr\}  dt \\
	&+ \frac{2H(t,\lambda(t)) ( \mu_0 -  r ) Y^*(t) }{ G(t)    \sigma_0 }    dW_0(t) -  \int_{\mathbb{R}_{>0}} z \frac{2H(t,\lambda(t))\kappa_r \mu_1 Y^*(t-)}{ G(t)   \sigma_1^2} \widetilde{N}_1(dt,dz) \\
	&- \int_{\mathbb{R}_{>0}} z \frac{1}{G(t)}\biggl\{\frac{\iota_r  \mu_2 Y^*(t-) }{ {\bf \widetilde{H}}(z)} + \frac{2Y^*(t-) {\bf \widetilde{H}}(H_z) }{ {\bf \widetilde{H}}(z)} + \frac{{\bf \widetilde{H}}(I_z)}{ {\bf \widetilde{H}}(z)}\biggr\}  \widetilde{N}_2(dt,dz), \\
	\end{align*}
	and
	\begin{align*}
	dY^*(t) =& -Y^*(t) \frac{  \mu_0 -  r   }{\sigma_0} dW_0(t) + \int_{\mathbb{R}_{>0}} Y^*(t-) \frac{ \kappa_r \mu_1 z}{\sigma_1^2} \widetilde{N}_1(dt,dz) \\
	&- \int_{\mathbb{R}_{>0}} \frac{1 }{2H(t,\lambda(t)+z)} \biggl\{\bigg(2  ({\bf 1}- \frac{z{\bf \widetilde{H}} }{ {\bf \widetilde{H}}(z)})(H_z) -  \frac{\iota_r   \mu_2  z}{ {\bf \widetilde{H}}(z)}\bigg) Y^*(t-) +  ({\bf 1}- \frac{z{\bf \widetilde{H}} }{ {\bf \widetilde{H}}(z)})(I_z)\biggr\} \widetilde{N}_2(dt,dz).
	\end{align*}
	
	\par By applying It\^{o}'s lemma to $H(t,\lambda)$ and $I(t,\lambda)$, we have
	\begin{align}
	\label{mmvSN.SDE.I}
	d I(t,\lambda(t)) =& \bigg( I_t - \delta \lambda(t) I_{\lambda} + \rho \int_{\mathbb{R}_{>0}} I_z F_2(dz) \bigg) dt + \int_{\mathbb{R}_{>0}} I_z \widetilde{N}_2(dt,dz),
	\end{align}
	and
	\begin{align}
	\label{mmvSN.SDE.H}
	d H(t,\lambda(t)) = \bigg( H_t - \delta \lambda(t) H_{\lambda} + \rho \int_{\mathbb{R}_{>0}} H_z F_2(dz) \bigg) dt + \int_{\mathbb{R}_{>0}} H_z \widetilde{N}_2(dt,dz).
	\end{align}
	
	\par By substituting \eqref{mmvSN.equation.function.value.coefficients} into \eqref{mmvSN.SDE.I} and \eqref{mmvSN.SDE.H},
	it follows that
	\begin{align*}
	d I(t,\lambda(t)) =&  \bigg( (  \kappa_r - \kappa)\mu_1\lambda  + ( \iota_r - \iota) k \mu_2\rho\bigg) G(t) dt - \frac{ {\bf \widetilde{H}}(I_z) \rho \iota_r \mu_2 }{ {\bf \widetilde{H}}(z)} dt \\
	&+ \rho \biggl\{ 2 {\bf \widetilde{H}}(\frac{H_z I_z}{z}) - 2 \frac{{\bf \widetilde{H}}(H_z) {\bf \widetilde{H}}(I_z)}{{\bf \widetilde{H}}(z)} \biggr\}  dt + \int_{\mathbb{R}_{>0}} I_z \widetilde{N}_2(dt,dz),
	\end{align*}
	and
	\begin{align*}
	&d \bigg( Y^*(t) H(t,\lambda(t)) \bigg) \\
	=& H(t,\lambda(t)) d Y^*(t) + Y^*(t) d H(t,\lambda(t)) +  d \bigg[H(\cdot,\lambda(\cdot)),Y\bigg](t) \\
	=& - H(t,\lambda(t)) Y^*(t) \frac{  \mu_0 -  r   }{\sigma_0} dW_0(t) + \int_{\mathbb{R}_{>0}} H(t,\lambda(t-)) Y^*(t-) \frac{ \kappa_r \mu_1 z}{\sigma_1^2} \widetilde{N}_1(dt,dz) \\
	&- \int_{\mathbb{R}_{>0}} \frac{H(t,\lambda(t))}{2H(t,\lambda(t)+z)} \biggl\{\bigg(2  ({\bf 1}- \frac{z{\bf \widetilde{H}} }{ {\bf \widetilde{H}}(z)})(H_z) -  \frac{\iota_r   \mu_2  z}{ {\bf \widetilde{H}}(z)}\bigg) Y^*(t-) +  ({\bf 1}- \frac{z{\bf \widetilde{H}} }{ {\bf \widetilde{H}}(z)})(I_z)\biggr\} \widetilde{N}_2(dt,dz) \\
	&+ \bigg( H_t - \delta \lambda(t) H_{\lambda} + \rho \int_{\mathbb{R}_{>0}} H_z F_2(dz) \bigg) Y^*(t) dt + Y^*(t) \int_{\mathbb{R}_{>0}} H_z \widetilde{N}_2(dt,dz) \\
	&- \int_{\mathbb{R}_{>0}} \frac{H_z }{2H(t,\lambda(t)+z)} \biggl\{\bigg(2  ({\bf 1}- \frac{z{\bf \widetilde{H}} }{ {\bf \widetilde{H}}(z)})(H_z) -  \frac{\iota_r   \mu_2  z}{ {\bf \widetilde{H}}(z)}\bigg) Y^*(t-) +  ({\bf 1}- \frac{z{\bf \widetilde{H}} }{ {\bf \widetilde{H}}(z)})(I_z)\biggr\} N_2(dt,dz) \\ \\
	=& - H(t,\lambda(t)) Y^*(t) \frac{  \mu_0 -  r   }{\sigma_0} dW_0(t) + \int_{\mathbb{R}_{>0}} H(t,\lambda(t-)) Y^*(t-) \frac{ \kappa_r \mu_1 z}{\sigma_1^2} \widetilde{N}_1(dt,dz) \\
	&+ \bigg( H_t - \delta \lambda(t) H_{\lambda} + \rho \int_{\mathbb{R}_{>0}} H_z F_2(dz) \bigg) Y^*(t) dt + Y^*(t) \int_{\mathbb{R}_{>0}} H_z \widetilde{N}_2(dt,dz) \\
	&- \int_{\mathbb{R}_{>0}} \frac{1 }{2} \biggl\{\bigg(2  ({\bf 1}- \frac{z{\bf \widetilde{H}} }{ {\bf \widetilde{H}}(z)})(H_z) -  \frac{\iota_r   \mu_2  z}{ {\bf \widetilde{H}}(z)}\bigg) Y^*(t-) +  ({\bf 1}- \frac{z{\bf \widetilde{H}} }{ {\bf \widetilde{H}}(z)})(I_z)\biggr\} \widetilde{N}_2(dt,dz) \\
	&- \rho \int_{\mathbb{R}_{>0}} \frac{H_z }{2H(t,\lambda(t)+z)} \biggl\{\bigg(2  ({\bf 1}- \frac{z{\bf \widetilde{H}} }{ {\bf \widetilde{H}}(z)})(H_z) -  \frac{\iota_r   \mu_2  z}{ {\bf \widetilde{H}}(z)}\bigg) Y^*(t-) +  ({\bf 1}- \frac{z{\bf \widetilde{H}} }{ {\bf \widetilde{H}}(z)})(I_z)\biggr\} F_2(dz)dt \\
	=& - H(t,\lambda(t)) Y^*(t) \frac{  \mu_0 -  r   }{\sigma_0} dW_0(t) + \int_{\mathbb{R}_{>0}} H(t,\lambda(t-)) Y^*(t-) \frac{ \kappa_r \mu_1 z}{\sigma_1^2} \widetilde{N}_1(dt,dz) \\
	&+  \frac{1 }{2}\int_{\mathbb{R}_{>0}} \biggl\{   2\frac{z{\bf \widetilde{H}}(H_z) }{ {\bf \widetilde{H}}(z)}Y^*(t-) +  \frac{\iota_r   \mu_2  z}{ {\bf \widetilde{H}}(z)} Y^*(t-) - I_z + \frac{z{\bf \widetilde{H}}(I_z) }{ {\bf \widetilde{H}}(z)}\biggr\} \widetilde{N}_2(dt,dz) \\
	&- \bigg( \frac{H(t,\lambda(t)) ( \mu_0 -  r )^2 }{    \sigma_0 ^2} + \frac{H(t,\lambda(t)) \lambda(t) \kappa_r^2  \mu_1^2 }{ \sigma_1^2} +  \frac{\rho \iota_r^2  \mu_2^2 }{ 2 {\bf \widetilde{H}}(z)} + \frac{ {\bf \widetilde{H}}(H_z) \rho \iota_r  \mu_2 }{ {\bf \widetilde{H}}(z)} \bigg) Y^*(t) dt\\
	&- \rho \biggl\{ {\bf \widetilde{H}}(\frac{H_z I_z}{z}) -  \frac{{\bf \widetilde{H}}(H_z){\bf \widetilde{H}}(I_z)}{{\bf \widetilde{H}}(z)} \biggr\} dt.
	\end{align*}
	
	\par The last equality is derived because
	\begin{align*}
	&\rho \int_{\mathbb{R}_{>0}} \frac{H_z}{2H(t,\lambda(t)+z)} \biggl\{\bigg(2  ({\bf 1}- \frac{z{\bf \widetilde{H}} }{ {\bf \widetilde{H}}(z)})(H_z) -  \frac{\iota_r   \mu_2  z}{ {\bf \widetilde{H}}(z)}\bigg) Y^*(t) +  ({\bf 1}- \frac{z{\bf \widetilde{H}} }{ {\bf \widetilde{H}}(z)})(I_z)\biggr\} F_2(dz) \\
	=& \rho \biggl\{ 2 Y^*(t) {\bf \widetilde{H}}(\frac{H_z^2}{z}) - 2 Y^*(t) \frac{{\bf \widetilde{H}}(H_z)^2}{{\bf \widetilde{H}}(z)} - Y^*(t) \frac{\iota_r \mu_2 {\bf \widetilde{H}}(H_z)}{{\bf \widetilde{H}}(z)} +  {\bf \widetilde{H}}(\frac{H_z I_z}{z}) -  \frac{{\bf \widetilde{H}}(H_z){\bf \widetilde{H}}(I_z)}{{\bf \widetilde{H}}(z)}\biggr\}.
	\end{align*}
	
	\par Therefore,
	\begin{align*}
	&d\biggl(X^*(t) G(t)\biggr) = G(t)dX^*(t) - G(t) r X^*(t) dt    \\
	=&  - \bigg(  (\kappa_r - \kappa)\mu_1\lambda(t)   + ( \iota_r - \iota)k\mu_2\rho\bigg) G(t) dt \\
	&+ \frac{2H(t,\lambda(t)) ( \mu_0 -  r )^2 Y^*(t) }{   \sigma_0 ^2}  dt +  \lambda(t) \frac{2H(t,\lambda(t))\kappa_r^2 \mu_1^2 Y^*(t)}{  \sigma_1^2} dt  \\
	&+ \iota_r \mu_2\rho \biggl\{\frac{\iota_r  \mu_2 Y^*(t) }{ {\bf \widetilde{H}}(z)} + \frac{2 Y^*(t) {\bf \widetilde{H}}(H_z) }{ {\bf \widetilde{H}}(z)} + \frac{{\bf \widetilde{H}}(I_z)}{ {\bf \widetilde{H}}(z)}\biggr\}  dt \\
	&+ \frac{2H(t,\lambda(t)) ( \mu_0 -  r ) Y^*(t) }{    \sigma_0 }    dW_0(t) -  \int_{\mathbb{R}_{>0}} z \frac{2H(t,\lambda(t-))\kappa_r \mu_1 Y^*(t-)}{   \sigma_1^2} \widetilde{N}_1(dt,dz) \\
	&- \int_{\mathbb{R}_{>0}} z \biggl\{\frac{\iota_r  \mu_2 Y^*(t-) }{ {\bf \widetilde{H}}(z)} + \frac{2Y^*(t-) {\bf \widetilde{H}}(H_z) }{ {\bf \widetilde{H}}(z)} + \frac{{\bf \widetilde{H}}(I_z)}{ {\bf \widetilde{H}}(z)}\biggr\}  \widetilde{N}_2(dt,dz) \\
	=& - 2 d \bigg( Y^*(t) H(t,\lambda(t)) \bigg) - d I(t,\lambda(t)).
	\end{align*}
	
\end{proof}

\subsection{Proof of Lemma \ref{mmvSN.lemma.coefficients}}

\begin{proof}

\par 

It is easy to obtain that 
\begin{align}
\label{mmvSN.equation.function.value.coefficients.suppose.G}
G(t) = e^{r(T-t)}.
\end{align}

\par Let us recall that $H(t,\lambda)$ satisfies
\begin{align*}
H_t -\delta \lambda H_{\lambda} +  \rho \int_{\mathbb{R}_{>0}} H_z F_2(dz) + \frac{H(t,\lambda) ( \mu_0 -  r )^2 }{    \sigma_0 ^2} + \frac{H(t,\lambda)  \lambda \kappa_r^2  \mu_1^2 }{ \sigma_1^2} + \frac{\rho \bigg(\iota_r  \mu_2 + 2 {\bf \widetilde{H}}(H_z) \bigg)^2 }{ 2 {\bf \widetilde{H}}(z)}  - 2 \rho  {\bf \widetilde{H}}(\frac{H_z^2}{z})  = 0.
\end{align*}
\par We hypothesize that it takes the following form
\begin{align}
\label{mmvSN.equation.function.value.coefficients.suppose.H}
H(t,\lambda) = \frac{1}{2 \theta}e^{\eta(t) \lambda + \zeta(t)}.
\end{align}

\par By substituting \eqref{mmvSN.equation.function.value.coefficients.suppose.H} into \eqref{mmvSN.equation.function.value.coefficients} and by separating the variables,
we obtain the following two ordinary differential equations
\begin{align*}
\begin{cases}
&\eta'(t)  -\delta  \eta(t)  +\frac{\kappa_r^2  \mu_1^2 }{ \sigma_1^2} = 0, \\
&\zeta'(t)   + \frac{ ( \mu_0 -  r )^2  }{    \sigma_0 ^2} \\
&+ \rho \bigg\{ \frac{((\iota_r+1)\mu_2 - \int_{\mathbb{R}_{>0}} z e^{-\eta(t) z} F_2(dz))^2}{\int_{\mathbb{R}_{>0}}  z^2 e^{-\eta(t) z} F_2(dz)} - \int_{\mathbb{R}_{>0}}  e^{-\eta(t) z} F_2(dz) + 1\bigg\} = 0.
\end{cases}
\end{align*}
The solutions are given as follows
\begin{align*}
\begin{cases}
\eta(t) =& \frac{\kappa_r^2   \mu_1^2}{\delta \sigma_1^2}(1-e^{-\delta(T-t)}), \\
\zeta(t) =& \rho \int_t^T  \biggl\{ \frac{((\iota_r+1)\mu_2 - \int_{\mathbb{R}_{>0}} z e^{-\eta(s) z} F_2(dz))^2}{\int_{\mathbb{R}_{>0}}  z^2 e^{-\eta(s) z} F_2(dz)} - \int_{\mathbb{R}_{>0}}  e^{-\eta(s) z} F_2(dz)  \biggr\} ds \\
& + (\rho  + \frac{ ( \mu_0 -  r )^2  }{    \sigma_0 ^2})(T-t).
\end{cases}
\end{align*}


\par For $I(t,\lambda)$, we suppose it has the following form
\begin{align}
\label{mmvSN.equation.function.value.coefficients.suppose.I}
I(t,\lambda) = \alpha(t) \lambda + \beta(t).
\end{align}

\par Substituting \eqref{mmvSN.equation.function.value.coefficients.suppose.G} and \eqref{mmvSN.equation.function.value.coefficients.suppose.I} into \eqref{mmvSN.equation.function.value.coefficients} gives
%
\begin{align*}
\begin{cases}
\alpha'(t) - \delta \alpha(t) - (\kappa_r - \kappa)\mu_1 e^{r(T-t)} = 0, \\
\beta'(t) + \rho \alpha(t) \mu_2 (\iota_r + 1)  - ( \iota_r - \iota) k\mu_2\rho e^{r(T-t)} = 0,
\end{cases}
\end{align*}
with solutions as follows
\begin{align*}
\begin{cases}
\alpha(t) = - (\kappa_r - \kappa)\mu_1 \int_t^T e^{-\delta(s-t)} e^{r(T-s)} ds , \\
\beta(t) = \rho \int_t^T  \biggl\{\alpha(s) \mu_2 (\iota_r + 1) - ( \iota_r - \iota) k\mu_2 e^{r(T-s)}\biggr\} ds,
\end{cases}
\end{align*}
that is,
\begin{align*}
\begin{cases}
\alpha(t) = \frac{(\kappa_r - \kappa)\mu_1}{\delta + r} (e^{- \delta(T-t)} - e^{r(T-t)})  , \\
\beta(t) = \frac{\rho (\kappa_r - \kappa)  (\iota_r + 1) \mu_1 \mu_2}{\delta + r} \bigg( \frac{1}{\delta}(1-e^{-\delta(T-t)})+\frac{1}{r}(1-e^{r(T-t)}) \bigg)  - \frac{\rho ( \iota_r - \iota) k\mu_2}{r} (1-e^{r(T-t)}).
\end{cases}
\end{align*}

\par Since 
\begin{align*}
\frac{\rho}{2} \biggl\{ {\bf \widetilde{H}}(\frac{I_z^2}{z}) - \frac{{\bf \widetilde{H}}(I_z)^2}{{\bf \widetilde{H}}(z)} \biggr\} = 0,
\end{align*}
and
\begin{align*}
K(T,\lambda) = 0,
\end{align*}
we have
\begin{align*}
K(t,\lambda) = 0, \quad \forall \quad t \in [0,T].
\end{align*}

\end{proof}

\subsection{Proof of Theorem \ref{mmvSN.theorem.solution.explicit}}

\begin{proof}
	We first prove that the strategies $a^*(t)$ and $b^*(t,z)$ defined by \eqref{mmvSN.strategy.optimal.a} and \eqref{mmvSN.strategy.optimal.b} belong to $\mathcal{A}[s,T]$ and $\mathcal{B}[s,T]$, respectively. By Theorem 13 (2) and Lemma 7 (2) of \cite{kabanov1979absolute}, the SDE \eqref{mmvSN.dynamic.y} has a unique solution in the class of nonnegative local martingales which we also denote by $Y$. Since $b^*(t,z)$ is deterministic, bounded and satisfies
	\begin{align}
	p^*(t,z) \geq -1, \\
	q^*(t,z) \geq -1,
	\end{align}
	$Y$ is a square integrable martingale. Consequently, $b^*(t,z) \in \mathcal{B}[s,T]$.
	\par On the other hand, since
	\begin{equation*}
	\sup_{s \leq t \leq T} \mathbb{E}^PY(t)^2 < +\infty,
	\end{equation*}
	it holds that
	\begin{equation*}
	\mathbb{E}^P\left[ \int_s^T \pi^*(t,x,y,\lambda)^2 + u^*(t,x,y,\lambda)^2 + v^*(t,x,y,\lambda)^2dt\right] < +\infty,
	\end{equation*}
	which proves that $a^*(t) \in \mathcal{A}[s,T]$.
	\par We then verify that \eqref{mmvSN.value.function.W} is the value function of Problem \eqref{obj.3}. By applying It{\^o}'s lemma to $W(t,x,y,\lambda)$, we have
	\begin{align*}
	&W(T,X(T),Y(T),\lambda(T)) = W(s,x,y,\lambda) + \int_s^T \mathcal{L}_t^{a^*,b^*} W(t,X(t),Y(t),\lambda(t)) dt \\
	&+ \int_s^T \pi^*(t) \sigma_0 \frac{\partial W}{\partial x} d W_0(s) + \int_s^T Y(t) o^*(t) \frac{\partial W}{\partial y} d W_0(s)\\
	&+ \int_s^T \int_{\mathbb{R}_{>0}} W(t,X(t-)-u^*(t)z,Y(t-)(1+p^*(t,z)),\lambda(t-))-W(t,X(t-),Y(t-),\lambda(t-)) \widetilde{N}_1(dt,dz) \\
	&+ \int_s^T \int_{\mathbb{R}_{>0}} W(t,X(t-)-kv^*(t)z,Y(t-)(1+q^*(t,z)),\lambda(t-)+z)-W(t,X(t-),Y(t-),\lambda(t-)) \widetilde{N}_2(dt,dz).
	\end{align*}
	\par Let $\tau_B = \inf\bigg\{t > 0: |\pi^*(t) \sigma_0 \frac{\partial W}{\partial x}| \geq B \bigg\} \wedge \inf\bigg\{t > 0: |Y(t) o^*(t)| \geq B \bigg\} \\
	\wedge\inf\bigg\{t > 0: \int_s^t \int_{\mathbb{R}_{>0}} |W(t,X(r-)-u(r)z,Y(r-)(1+p^*(r,z)),\lambda(t-))| \lambda(t) F_1(dz)dr \geq B \bigg\} \\
	\wedge \inf\bigg\{t > 0: \int_s^t \int_{\mathbb{R}_{>0}} |W(t,X(r-)-u(r)z,Y(r-)(1+q^*(r,z)),\lambda(t-)+z)| \rho F_2(dz)dr \geq B \bigg\}$. We have
	\begin{align*}
	&\mathbb{E}^P \left[ W(T\vee \tau_B,X(T\vee \tau_B),Y(T\vee \tau_B),\lambda(T\vee \tau_B)) \right] \\
	&= W(s,x,y,\lambda) + \mathbb{E}^P \int_s^{T \vee \tau_B} \mathcal{L}_t^{a^*,b^*} W(t,X(t),Y(t),\lambda(t)) dt.
	\end{align*}
	\par Since $W$ satisfies the HJBI equation, we obtain that
	\begin{align*}
	\begin{cases}
	&\mathbb{E}^P \left[ W(T\vee \tau_B,X^{a^*}(T\vee \tau_B),Y^{b^*}(T\vee \tau_B),\lambda(T\vee \tau_B)) \right] = W(s,x,y,\lambda), \\
	&\mathbb{E}^P \left[ W(T\vee \tau_B,X^{a^*}(T\vee \tau_B),Y^{b}(T\vee \tau_B),\lambda(T\vee \tau_B)) \right] \geq W(s,x,y,\lambda), \\
	&\mathbb{E}^P \left[ W(T\vee \tau_B,X^{a}(T\vee \tau_B),Y^{b^*}(T\vee \tau_B),\lambda(T\vee \tau_B)) \right] \leq W(s,x,y,\lambda).
	\end{cases}
	\end{align*}
	\par By letting $B \to \infty$, we obtain
	\begin{equation*}
	J^{a,b^*}(s,x,y,\lambda) \leq J^{a^*,b^*}(s,x,y) = W(s,x,y,\lambda) \leq J^{a^*,b}(s,x,y,\lambda).
	\end{equation*}
	\par Consequently,
	\begin{equation*}
	W(s,x,y,\lambda) = \sup_{(\pi,u,v) \in \mathcal{A}[s,T]} \left ( \inf_{ (o,p,q) \in \mathcal{B}[s,T] } J^{a,b}(s,x,y,\lambda) \right ) = \inf_{(o,p,q) \in \mathcal{B}[s,T] } \left ( \sup_{(\pi,u,v) \in \mathcal{A}[s,T]} J^{a,b}(s,x,y,\lambda) \right ),
	\end{equation*}
	and $(a^*(t), b^*(t,z))$ is a Nash equilibrium of Problem \eqref{obj.3}.
\end{proof}

\subsection{Proof of Theorem \ref{mmfSN.theorem.efficient_frontier}}
\begin{proof}
\par By Theorem \ref{mmvSN.lemma.relationship_of_XY} and Lemma \ref{mmvSN.lemma.coefficients}, to obtain $\mathbb{V}ar_{s,x,y,\lambda}^PX^*(t)$, we only need to compute the values of
\begin{align*}
\begin{cases}
\mathbb{E}_{s,x,y,\lambda}^P \lambda(t), \\
\mathbb{E}_{s,x,y,\lambda}^P \lambda(t)^2, \\
\mathbb{E}_{s,x,y,\lambda}^P [Y^*(t) H(t,\lambda(t))], \\
\mathbb{E}_{s,x,y,\lambda}^P [Y^*(t)^2 H(t,\lambda(t))^2], \\
\mathbb{E}_{s,x,y,\lambda}^P [Y^*(t) H(t,\lambda(t)) \lambda(t)].
\end{cases}
\end{align*}


%
%

\par 1) $\mathbb{E}_{s,x,y,\lambda}^P \lambda(t)$ and $\mathbb{E}_{s,x,y,\lambda}^P \lambda(t)^2$: By Feyman-Kac theorem, $\mathbb{E}_{s,x,y,\lambda}^P \lambda(t)$ and $\mathbb{E}_{s,x,y,\lambda}^P \lambda(t)^2$ are the solutions of the following PDE
\begin{align*}
L_{s} - \delta \lambda L_{\lambda} + \rho \int_{\mathbb{R}_{>0}} L_{z}  F_2(dz)  = 0
\end{align*}
with terminal values $L(t,\lambda) = \lambda$ and $L(t,\lambda) = \lambda^2$, respectively. It is easy to obtain that
\begin{align*}
\begin{cases}
\mathbb{E}_{s,x,y,\lambda}^P \lambda(t) = e^{-\delta(t-s)}\lambda + \frac{\rho \mu_2}{\delta} (1 - e^{-\delta(t-s)}), \\
\mathbb{E}_{s,x,y,\lambda}^P \lambda(t)^2 = \bigg( e^{-\delta(t-s)}\lambda + \frac{\rho \mu_2}{\delta} (1 - e^{-\delta(t-s)}) \bigg)^2 + \frac{\rho \sigma_2^2}{2 \delta} (1 - e^{-2\delta(t-s)}).
\end{cases}
\end{align*}


\par 2) $\mathbb{E}_{s,x,y,\lambda}^P [ Y^*(t) H(t,\lambda(t)) ]$: To get this value, we first need to find the SDE that $H(t,\lambda(t))$ satisfies. By applying It{\^o}'s lemma to $H(t,\lambda(t))$, it follows that
\begin{align*}
d H(t,\lambda(t)) = \bigg( H_t - \delta \lambda(t) H_{\lambda} + \rho \int_{\mathbb{R}_{>0}} H_z F_2(dz) \bigg) dt + \int_{\mathbb{R}_{>0}} H_z \widetilde{N}_2(dt,dz).
\end{align*}
\par Note that
\begin{align*}
\frac{1}{2  H(t,\lambda)}\bigg(\frac{\iota_r  \mu_2 }{  {\bf \widetilde{H}}(z)} + \frac{2 {\bf \widetilde{H}}(H_z) }{  {\bf \widetilde{H}}(z)}\bigg) = \phi(t),
\end{align*}
and by \eqref{mmvSN.equation.function.value.coefficients}, we have
\begin{align*}
d H(t,\lambda(t)) =& -\bigg( \frac{H(t,\lambda(t)) ( \mu_0 -  r )^2 }{    \sigma_0 ^2} + \frac{H(t,\lambda(t))  \lambda(t) \kappa_r^2  \mu_1^2 }{ \sigma_1^2} + 2 \rho \bigg( H(t,\lambda(t))^2 \phi(t)^2 {\bf \widetilde{H}}(z) -  {\bf \widetilde{H}}(\frac{H_z^2}{z})\bigg) \bigg) dt \\
&+ \int_{\mathbb{R}_{>0}} H_z \widetilde{N}_2(dt,dz).
\end{align*}

\par Again by It{\^o}'s lemma, we have
\begin{align*}
d Y^*(t) H(t,\lambda(t)) 
=& - Y^*(t) H(t,\lambda(t)) \biggl\{  \frac{ ( \mu_0 -  r )^2 }{    \sigma_0 ^2} + \frac{\lambda(t) \kappa_r^2  \mu_1^2 }{ \sigma_1^2} + \rho \iota_r \mu_2 \phi(t)\biggr\}dt \\
&+ Y^*(t) H(t,\lambda(t)) \biggl\{ - \frac{  \mu_0 -  r   }{\sigma_0} dW_0(t) + \int_{\mathbb{R}_{>0}} \frac{ \kappa_r \mu_1 z}{\sigma_1^2} \widetilde{N}_1(dt,dz) + \int_{\mathbb{R}_{>0}} \phi(t) z \widetilde{N}_2(dt,dz) \biggr\}.
\end{align*}

\par Denote by
\begin{align*}
U_1(t) =&  \frac{Y^*(t) H(t,\lambda(t))}{y H(s,\lambda)}, \\
A_1(t) =& -  \int_s^t \biggl\{  \frac{ ( \mu_0 -  r )^2 }{    \sigma_0 ^2} + \frac{\lambda(u) \kappa_r^2  \mu_1^2 }{ \sigma_1^2} + \rho \iota_r \mu_2 \phi(u)\biggr\}du, \\
M_1(t) =&  - \int_s^t \frac{  \mu_0 -  r   }{\sigma_0} dW_0(u) + \int_s^t \int_{\mathbb{R}_{>0}} \frac{ \kappa_r \mu_1 z}{\sigma_1^2} \widetilde{N}_1(du,dz) + \int_s^t \int_{\mathbb{R}_{>0}} \phi(u) z \widetilde{N}_2(du,dz).
\end{align*}
Then $U_1(t)$ satisfies
\begin{align*}
d U_1(t) = U_1(t) d A_1(t) + U_1(t) d M_1(t), \quad U_1(s) = 1.
\end{align*}
\par Note that $e^{-A_1(t)  } U_1(t)$ is a stochastic exponential satisfying
\begin{align*}
   e^{-A_1(t)  } U_1(t)  =  1 + \int_s^t e^{-A_1(u) } U_1(u) d M_1(u),
\end{align*}
we can define a new probability measure $P_1$ by
\begin{align*}
  \frac{d P_1}{d P} \bigg|_{\mathcal{F}_t} = e^{-A_1(t) } U_1(t).
\end{align*}
Under the probability measure $P_1$, the following stochastic measure is a martingale
\begin{align*}
\widetilde{N}_2^{P_1}(dt,dz) = \widetilde{N}_2(dt,dz) - \phi(t)z \rho F_2(dz) dt,
\end{align*}
such that
\begin{align*}
\mathbb{E}_{s,x,y,\lambda}^{P_1} N_2(t) = \int_s^t \int_{\mathbb{R}_{>0}} (\phi(u)z +1) \rho F_2(dz) du.
\end{align*}
Now the representation of $\lambda$ under the new probability $P_1$ is as follows
\begin{align*}
d \lambda(t) = ( -\delta \lambda(t) + \rho \mu_2 + \rho \sigma_2^2 \phi(t)) dt + \int_{\mathbb{R}_{>0}} z \widetilde{N}_2^{P_1}(dt,dz).
\end{align*}

\par By Feyman-Kac theorem,
\begin{align*}
L^{YH}(s,\lambda) :=  \mathbb{E}_{s,x,y,\lambda}^{P_1} e^{ - \int_s^t \{  \frac{ ( \mu_0 -  r )^2 }{    \sigma_0 ^2} + \frac{\lambda(u) \kappa_r^2  \mu_1^2 }{ \sigma_1^2} + \rho \iota_r \mu_2 \phi(u) \}du},
\end{align*}
is the unique probability solution of the following PDE,
\begin{align}
\label{mmvSN.equation.L.yh}
\begin{cases}
L^{YH}_{s} - \delta \lambda L^{YH}_{\lambda} + \rho \int_{\mathbb{R}_{>0}} L^{YH}_{z} (\phi(s)z + 1) F_2(dz) -  \biggl\{  \frac{ ( \mu_0 -  r )^2 }{    \sigma_0 ^2} + \frac{\lambda \kappa_r^2  \mu_1^2 }{ \sigma_1^2} + \rho \iota_r \mu_2 \phi(s)\biggr\}  L^{YH}(s,\lambda) = 0, \\
L^{YH}(t,\lambda) = 1.
\end{cases}
\end{align}
We solve it to obtain
\begin{align}
\label{mmvSN.equation.L.yh.solution}
L^{YH}(s,\lambda) = e^{\psi_1(s,t) \lambda + \psi_2(s,t)} e^{-\eta(s) \lambda - \zeta(s)} = \frac{1}{2\theta} e^{\psi_1(s,t) \lambda + \psi_2(s,t)} H(s,\lambda)^{-1},
\end{align}
where
\begin{align*}
\begin{cases}
\psi_1(s,t) = \frac{\kappa_r^2   \mu_1^2}{\delta \sigma_1^2}(e^{-\delta(t-s)}-e^{-\delta(T-s)}), \\
\psi_2(s,t) = \zeta(t) + \rho \int_s^t \int_{\mathbb{R}_{>0}}  e^{-\eta(u) z} (e^{\psi_1(u,t) z} - 1) (\phi(u) z + 1) F_2(dz) du.
\end{cases}
\end{align*}


\par Therefore,
\begin{align*}
\mathbb{E}_{s,x,y,\lambda}^P [ Y^*(t) H(t,\lambda(t)) ] =& y H(s,\lambda) \mathbb{E}_{s,x,y,\lambda}^P U_1(t) \\
=& y H(s,\lambda) \mathbb{E}_{s,x,y,\lambda}^P [e^{A_1(t)} e^{-A_1(t)} U_1(t)] \\
=& y H(s,\lambda)  L^{YH}(s,\lambda) \\
=& \frac{y}{2\theta} e^{\psi_1(s,t) \lambda + \psi_2(s,t)}.
\end{align*}

\par 3) $\mathbb{E}_{s,x,y,\lambda}^P [ Y^*(t)^2 H(T,\lambda(t))^2 ]$: We use the similar method as shown in 2). First by It{\^o}'s lemma,
\begin{align*}
d \bigg(Y^*(t) H(t,\lambda(t))\bigg)^2 
=&  \bigg(Y^*(t) H(t,\lambda(t))\bigg)^2 \biggl\{  -\frac{ ( \mu_0 -  r )^2 }{    \sigma_0 ^2} - \frac{\lambda(t) \kappa_r^2  \mu_1^2 }{ \sigma_1^2} - 2 \rho \iota_r \mu_2 \phi(t)  + \rho \sigma_2^2 \phi(t)^2  \biggr\}dt \\
&+ \bigg(Y^*(t) H(t,\lambda(t))\bigg)^2 \biggl\{ - \frac{  2 (\mu_0 -  r )  }{\sigma_0} dW_0(t) + \int_{\mathbb{R}_{>0}} (\frac{ 2 \kappa_r \mu_1 z}{\sigma_1^2} + \frac{ \kappa_r^2 \mu_1^2 z^2}{\sigma_1^4}) \widetilde{N}_1(dt,dz) \\
&+ \int_{\mathbb{R}_{>0}} (2 \phi(t) z + \phi(t)^2 z^2) \widetilde{N}_2(dt,dz) \biggr\}.
\end{align*}

\par Let
\begin{align*}
U_2(t) =&  \frac{Y^*(t)^2 H(t,\lambda(t))^2}{y^2 H(s,\lambda)^2}, \\
A_2(t) =& \int_s^t \biggl\{  -\frac{ ( \mu_0 -  r )^2 }{    \sigma_0 ^2} - \frac{\lambda(u) \kappa_r^2  \mu_1^2 }{ \sigma_1^2} - 2 \rho \iota_r \mu_2 \phi(u)  + \rho \phi(u)^2 \sigma_2^2 \biggr\}du, \\
M_2(t) =&  - \int_s^t \frac{  2 (\mu_0 -  r )  }{\sigma_0} dW_0(u) + \int_s^t \int_{\mathbb{R}_{>0}} (\frac{ 2 \kappa_r \mu_1 z}{\sigma_1^2} + \frac{ \kappa_r^2 \mu_1^2 z^2}{\sigma_1^4}) \widetilde{N}_1(du,dz) + \int_s^t \int_{\mathbb{R}_{>0}} (2 \phi(u) z + \phi(u)^2 z^2) \widetilde{N}_2(du,dz).
\end{align*}
Thus $U_2(t)$ satisfies
\begin{align*}
d U_2(t) = U_2(t) d A_2(t) + U_2(t) d M_2(t), \quad U_1(s) = 1,
\end{align*}
and
\begin{align*}
e^{-A_2(t)  } U_2(t)  =  1 + \int_s^t e^{-A_2(u) } U_2(u) d M_2(u).
\end{align*}
Define a new probability measure $P_2$ by
\begin{align*}
\frac{d P_2}{d P} \bigg|_{\mathcal{F}_t} = e^{-A_2(t) } U_2(t).
\end{align*}
The compensated Poisson process of $N_2$ under $P_2$ is given by
\begin{align*}
\widetilde{N}_2^{P_2}(dt,dz) = \widetilde{N}_2(dt,dz) - (2 \phi(t) z + \phi(t)^2 z^2) \rho F_2(dz) dt.
\end{align*}
Moreover,
\begin{align*}
\mathbb{E}_{s,x,y,\lambda}^{P_2} N_2(t) = \int_s^t \int_{\mathbb{R}_{>0}} (\phi(u)z +1)^2 \rho F_2(dz) du.
\end{align*}
Hence the representation of $\lambda$ under $P_2$ is as follows
\begin{align*}
d \lambda(t) = ( -\delta \lambda(t) + \rho \mu_2 + 2 \rho \sigma_2^2 \phi(t) + \rho \phi(t)^2 \int_{\mathbb{R}_{>0}} z^3 F_2(dz)) dt + \int_{\mathbb{R}_{>0}} z \widetilde{N}_2^{P_2}(dt,dz).
\end{align*}
By Feyman-Kac theorem,
\begin{align*}
L^{(YH)^2}(s,\lambda) :=  \mathbb{E}_{s,x,y,\lambda}^{P_2} e^{\int_s^t \{  -\frac{ ( \mu_0 -  r )^2 }{    \sigma_0 ^2} - \frac{\lambda(u) \kappa_r^2  \mu_1^2 }{ \sigma_1^2} - 2 \rho \iota_r \mu_2 \phi(u)  + \rho \sigma_2^2 \phi(u)^2 \}du}
\end{align*}
is the unique probability solution of the following PDE
\begin{align*}
\begin{cases}
L^{(YH)^2}_{s} - \delta \lambda L^{(YH)^2}_{\lambda} + \rho \int_{\mathbb{R}_{>0}} L^{(YH)^2}_{z} (\phi(s)z + 1)^2 F_2(dz) \\
+ \biggl\{ -\frac{ ( \mu_0 -  r )^2 }{    \sigma_0 ^2} - \frac{\lambda \kappa_r^2  \mu_1^2 }{ \sigma_1^2} - 2 \rho \iota_r \mu_2 \phi(s)  + \rho \sigma_2^2 \phi(s)^2 \biggr\}  L^{(YH)^2}(s,\lambda) = 0, \\
L^{(YH)^2}(t,\lambda) = 1,
\end{cases}
\end{align*}
in other words,
\begin{align*}
L^{(YH)^2}(s,\lambda) = e^{\psi_1(s,t) \lambda + \psi_3(s,t)} e^{-\eta(s) \lambda - \zeta(s)} = \frac{1}{2\theta} e^{\psi_1(s,t) \lambda + \psi_3(s,t)} H(s,\lambda)^{-1},
\end{align*}
where
\begin{align*}
\psi_3(s,t) =  \zeta(t) + \rho \int_s^t \int_{\mathbb{R}_{>0}}  e^{-\eta(u) z} (e^{\psi_1(u,t) z} - 1) (\phi(u) z + 1)^2 F_2(dz) du.
\end{align*}
\par Therefore
\begin{align*}
\mathbb{E}_{s,x,y,\lambda}^P [ Y^*(t)^2 H(t,\lambda(t))^2 ] =& y^2 H(s,\lambda)^2 \mathbb{E}_{s,x,y,\lambda}^P U_2(t) \\
=& y^2 H(s,\lambda)^2 \mathbb{E}_{s,x,y,\lambda}^P [e^{A_2(t)} e^{-A_2(t)} U_2(t)] \\
=& y^2 H(s,\lambda)^2  L^{(YH)^2}(s,\lambda) \\
=& \frac{y^2}{2\theta} e^{\psi_1(s,t) \lambda + \psi_3(s,t)} H(s,\lambda).
\end{align*}


\par 4) $\mathbb{E}_{s,x,y,\lambda}^P [ Y^*(t) H(t,\lambda(t)) \lambda(t) ]$: It can also be obtained by using Feyman-Kac theorem that
\begin{align}
L^{\lambda YH}(s,\lambda) :=  \mathbb{E}_{s,x,y,\lambda}^{P_1} [\lambda(t) e^{-  \int_s^t \{  \frac{ ( \mu_0 -  r )^2 }{    \sigma_0 ^2} + \frac{\lambda(u) \kappa_r^2  \mu_1^2 }{ \sigma_1^2} + \rho \iota_r \mu_2 \phi(u) \}du}]
\end{align}
is the unique probability solution of the following PDE
\begin{align}
\label{mmvSN.equation.L.lyh}
\begin{cases}
L^{\lambda YH}_{s} - \delta \lambda L^{\lambda YH}_{\lambda} + \rho \int_{\mathbb{R}_{>0}} L^{\lambda YH}_{z} (\phi(s)z + 1) F_2(dz) -  \biggl\{  \frac{ ( \mu_0 -  r )^2 }{    \sigma_0 ^2} + \frac{\lambda \kappa_r^2  \mu_1^2 }{ \sigma_1^2} + \rho \iota_r \mu_2 \phi(s)\biggr\}  L^{\lambda YH}(s,\lambda) = 0 \\
L^{\lambda YH}(t,\lambda) = \lambda.
\end{cases}
\end{align}
\par To solve the PDE, we suppose the form of $L^{\lambda YH}(s,\lambda) := L^{YH}(s,\lambda) \widetilde{L}(s,\lambda)$. Recall that $L^{YH}(s,\lambda)$ satisfies \eqref{mmvSN.equation.L.yh} and \eqref{mmvSN.equation.L.yh.solution}. By substituting $L^{YH}(s,\lambda) \widetilde{L}(s,\lambda)$ into \eqref{mmvSN.equation.L.lyh}, we see $\widetilde{L}(s,\lambda)$ satisfies
\begin{align*}
\begin{cases}
\widetilde{L}_{s} - \delta \lambda \widetilde{L}_{\lambda} + \rho \int_{\mathbb{R}_{>0}} \widetilde{L}_{z} e^{-\eta(s)z} (\phi(s)z + 1) F_2(dz)  = 0, \\
\widetilde{L}(t,\lambda) = \lambda,
\end{cases}
\end{align*}
which gives
\begin{align*}
\widetilde{L}(s,\lambda) = e^{-\delta(t-s) } \lambda +  \frac{\rho (\iota_r +1) \mu_2}{\delta} (1-e^{-\delta(t-s)} ).
\end{align*}
Hence,
\begin{align*}
L^{\lambda YH}(s,\lambda) =& \bigg( e^{-\delta(t-s) } \lambda +  \frac{\rho (\iota_r +1) \mu_2}{\delta} (1-e^{-\delta(t-s)} ) \bigg) e^{\psi_1(s,t) \lambda + \psi_2(s,t)} e^{-\eta(s) \lambda - \zeta(s)} \\
=& \frac{1}{2\theta}  \bigg( e^{-\delta(t-s) } \lambda +  \frac{\rho (\iota_r +1) \mu_2}{\delta} (1-e^{-\delta(t-s)} ) \bigg) e^{\psi_1(s,t) \lambda + \psi_2(s,t)} H(s,\lambda)^{-1}.
\end{align*}
\par Therefore
\begin{align*}
\mathbb{E}_{s,x,y,\lambda}^P [ Y^*(t) H(t,\lambda(t)) \lambda(t) ] =& y H(s,\lambda) \mathbb{E}_{s,x,y,\lambda}^P [ \lambda(t) U_1(t) ] \\
=& y H(s,\lambda) \mathbb{E}_{s,x,y,\lambda}^P [ \lambda(t) e^{A_1(t)} e^{-A_1(t)} U_1(t)] \\
=& y H(s,\lambda)  L^{\lambda YH}(s,\lambda) \\
=& \frac{y}{2\theta} e^{\psi_1(s,t) \lambda + \psi_2(s,t)} \bigg( e^{-\delta(t-s) } \lambda +  \frac{\rho (\iota_r +1) \mu_2}{\delta} (1-e^{-\delta(t-s)} ) \bigg).
\end{align*}


\par As a result, we obtian
\begin{align*}
&\mathbb{V}ar_{s,x,y,\lambda}^P[Y^*(t) H(t,\lambda(t))] = \frac{y^2}{4\theta^2}\bigg( 2 \theta e^{\psi_1(s,t) \lambda + \psi_3(s,t)} H(s,\lambda) -  e^{2\psi_1(s,t) \lambda + 2\psi_2(s,t)}\bigg), \\
&\mathbb{V}ar_{s,x,y,\lambda}^P[I(t,\lambda(t))] = \alpha(t)^2 \frac{\rho \sigma_2^2}{2 \delta} (1 - e^{-2\delta(t-s)}), \\
&Cov_{s,x,y,\lambda}^P [Y^*(t) H(t,\lambda(t)), I(t,\lambda(t))] = \alpha(t) \frac{y}{2\theta} e^{\psi_1(s,t) \lambda + \psi_2(s,t)} \frac{\rho \iota_r \mu_2}{\delta} (1-e^{-\delta(t-s)} ).
\end{align*}
Therefore
\begin{align}
\nonumber \mathbb{V}ar_{s,x,y,\lambda}^P[X^*(t) G(t)] 
\nonumber 
\nonumber =& \frac{y^2}{\theta^2}\bigg( 2 \theta e^{\psi_1(s,t) \lambda + \psi_3(s,t)} H(s,\lambda) -  e^{2\psi_1(s,t) \lambda + 2\psi_2(s,t)}\bigg) + \alpha(t)^2 \frac{\rho \sigma_2^2}{2 \delta} (1 - e^{-2\delta(t-s)}) \\
\label{mmvSN.equation.VarX} &+ 2  \alpha(t) \frac{y}{\theta} e^{\psi_1(s,t) \lambda + \psi_2(s,t)} \frac{\rho \iota_r \mu_2}{\delta} (1-e^{-\delta(t-s)} ). 
\end{align}


\par Moreover, by Theorem \ref{mmvSN.lemma.relationship_of_XY} and Lemma \ref{mmvSN.lemma.coefficients}, we have
\begin{align*}
&\mathbb{E}_{s,x,y,\lambda}^P [X^*(t) G(t)] - \mathbb{E}_{s,x,y,\lambda}^{Q^*} [X^*(t) G(t)] \\
=&   2 \mathbb{E}_{s,x,y,\lambda}^{Q^*} [ Y^*(t) H(t,\lambda(t)) ] - 2 \mathbb{E}_{s,x,y,\lambda}^P [ Y^*(t) H(t,\lambda(t)) ] +  \alpha(t) \mathbb{E}_{s,x,y,\lambda}^{Q^*} \lambda(t) - \alpha(t) \mathbb{E}_{s,x,y,\lambda}^P\lambda(t).
\end{align*}
\par Since $\mathbb{E}_{s,x,y,\lambda}^P [ Y^*(t) H(t,\lambda(t)) ]$ and $\mathbb{E}_{s,x,y,\lambda}^P\lambda(t)$ are derived in the first part of the proof, we only need to compute the values of $\mathbb{E}_{s,x,y,\lambda}^{Q^*} [ Y^*(t) H(t,\lambda(t)) ]$ and $\mathbb{E}_{s,x,y,\lambda}^{Q^*} \lambda(t)$.

\par Denote by
\begin{align*}
&dW_0^{Q^*}(t) := dW_0(t) +  \frac{\mu_0 - r}{\sigma_0} dt, \\ 
&\widetilde{N}_1^{Q^*}(dt,dz) := \widetilde{N}_1(dt,dz) - \frac{\lambda(t) \mu_1 \kappa_r z}{\sigma_1^2} F_1(dz)dt, \\
&\widetilde{N}_2^{Q^*}(dt,dz) := \widetilde{N}_2(dt,dz) - \rho \bigg(e^{-\eta(t)z}+\phi(t)e^{-\eta(t)z}z-1\bigg) F_2(dz)dt.
\end{align*}
\par It follows from Theorem 10 and Theorem 11 in \cite{kabanov1979absolute} that $W_0^{Q^*}$ is a standard Brownian motion, $\widetilde{N}_1^{Q^*}$ and $\widetilde{N}_2^{Q^*}$ are the compensated compound Poisson processes with respect to $N_1$ and $N_2$ under probability measure $Q^*$. Thus,
\begin{align*}
d Y^*(t) H(t,\lambda(t)) =& Y^*(t) H(t,\lambda(t)) \biggl\{ - \frac{  \mu_0 -  r   }{\sigma_0} dW_0^{Q^*}(t) + \int_{\mathbb{R}_{>0}} \frac{ \kappa_r \mu_1 z}{\sigma_1^2} \widetilde{N}_1^{Q^*}(dt,dz) + \int_{\mathbb{R}_{>0}} \phi(t) z \widetilde{N}_2^{Q^*}(dt,dz) \biggr\},
\end{align*}
and
\begin{align*}
d \lambda(t) = ( -\delta \lambda(t) + \rho (\iota_r + 1) \mu_2 ) dt + \int_{\mathbb{R}_{>0}} z \widetilde{N}_2^{Q^*}(dt,dz).
\end{align*}
Therefore, $Y^*(t) H(t,\lambda(t))$ is a martingale under $Q^*$ such that
\begin{align*}
\mathbb{E}_{s,x,y,\lambda}^{Q^*} [ Y^*(t) H(t,\lambda(t)) ] =  y H(s,\lambda).
\end{align*}
By the same method in 1), we also have
\begin{align*}
\mathbb{E}_{s,x,y,\lambda}^{Q^*} \lambda(t) = e^{-\delta(t-s)}\lambda + \frac{\rho (\iota_r + 1 ) \mu_2}{\delta} (1 - e^{-\delta(t-s)}).
\end{align*}
\par As a result,
\begin{align}
&\mathbb{E}_{s,x,y,\lambda}^P [X^*(t) G(t)] - \mathbb{E}_{s,x,y,\lambda}^{Q^*} [X^*(t) G(t)] 
\label{mmvSN.equation.EPX-EQX} =&   \frac{y}{\theta} \bigg(e^{\eta(s) \lambda + \zeta(s)} - e^{\psi_1(s,t) \lambda + \psi_2(s,t)}\bigg) +  \alpha(t) \frac{\rho \iota_r \mu_2}{\delta} (1 - e^{-\delta(t-s)}).
\end{align}

\par To complete the proof, we only need to substitute \eqref{mmvSN.equation.EPX-EQX} into \eqref{mmvSN.equation.VarX} and eliminate $\theta$.


\end{proof}

%
%
%


\section{Proofs of Statements in Section \ref{mmvSN.section.approximation}}

\subsection{Proof of Proposition \ref{mmvSN.diffusion.lemma.solution.suppose}}
\begin{proof}
	\par We suppose that the form of solution to \eqref{mmvSN.diffusion.hjbi.equation} is as follows
	\begin{align}
	\label{mmvSN.diffusion.function.value.suppose}
	W(t,x,y,\lambda) =  G(t)xy + H(t,\lambda)y^2 +I(t,\lambda)y.
	\end{align}

	\par Substituting \eqref{mmvSN.diffusion.function.value.suppose} into \eqref{mmvSN.diffusion.generator} gives
	\begin{align}
	\nonumber \mathcal{L}_t^{a,b} W(t,x,y,\lambda) =& G_t xy + H_t y^2 +I_t y +\biggl\{ r x  + (- \kappa_r + \kappa)\mu_1\lambda  + (- \iota_r + \iota) k \mu_2\rho \biggr\} G(t) y \\
	\nonumber &+ (-\delta \lambda + \rho \mu_2) ( H_{\lambda} y^2 +I_{\lambda} y ) + \frac{1}{2} \rho \sigma_2^2 ( H_{\lambda \lambda} y^2 +I_{\lambda \lambda} y )  \\
	\nonumber &+ \pi (\mu_0 - r)  G(t) y + \pi \sigma_0 y o G(t) + y^2 o^2 H(t,\lambda)\\
	\nonumber &+ \kappa_r u \mu_1\lambda G(t) y - u \sigma_1\sqrt{\frac{ \rho \mu_2}{\delta}} y p G(t) + y^2 p^2 H(t,\lambda) \\
	\label{mmvSN.diffusion.generator.suppose} &+ \iota_r k v \mu_2\rho G(t) y + \sigma_2\sqrt{ \rho}  y q (  2 H_{\lambda} y +I_{\lambda}  - G(t) k v )  + y^2 q^2 H(t,\lambda).
	\end{align}
	
	\par If $H(t,\lambda) > 0$ holds, we differentiate \eqref{mmvSN.diffusion.generator.suppose} with respect to $o$, $p$ and $q$ and let them equal $0$, respectively, then the minimum of \eqref{mmvSN.diffusion.generator.suppose} is attained at
	\begin{align*}
	o^*  =& - \frac{ G(t) \pi \sigma_0 }{2 H(t,\lambda) y}, \\
	p^*  =& \frac{ G(t)  u \sigma_1\sqrt{\frac{ \rho \mu_2}{\delta}}  }{2 H(t,\lambda)  y}, \\
	q^*  =& - \frac{\sigma_2\sqrt{ \rho}( 2 H_{\lambda} y +  I_{\lambda} - G(t) k v  ) }{ 2 H(t,\lambda) y}.
	\end{align*}	
	\par Plugging $o^*$, $p^*$ and $q^*$ back into \eqref{mmvSN.diffusion.generator.suppose} gives
	\begin{align}
	\nonumber \mathcal{L}_t^{a,b} W(t,x,y,\lambda) =& G_t xy + H_t y^2 +I_t y +\biggl\{ r x  + (- \kappa_r + \kappa)\mu_1\lambda  + ( - \iota_r + \iota)\mu_2\rho \biggr\} G(t) y \\
	\nonumber &+ (-\delta \lambda + \rho \mu_2) ( H_{\lambda} y^2 +I_{\lambda} y ) + \frac{1}{2} \rho \sigma_2^2 ( H_{\lambda \lambda} y^2 +I_{\lambda \lambda} y )  \\
	\nonumber &+ \pi (\mu_0 - r)  G(t) y - \frac{ G(t)^2 \pi^2 \sigma_0^2 }{4 H(t,\lambda)}\\
	\nonumber &+ \kappa_r u \mu_1\lambda G(t) y - \frac{ G(t)^2  u^2 \sigma_1^2 \frac{ \rho \mu_2}{\delta}  }{4 H(t,\lambda) } \\
	\label{mmvSN.diffusion.generator.suppose.2} &+ \iota_r k v \mu_2\rho G(t) y - \frac{\sigma_2^2  \rho ( 2 H_{\lambda} y +  I_{\lambda} - G(t) k v  )^2 }{ 4 H(t,\lambda) }.
	\end{align}
	
	\par If both $G(t)>0$ and $H(t,\lambda)>0$, by differentiating \eqref{mmvSN.diffusion.generator.suppose.2} with respect to $\pi$, $u$ and $v$ and letting it equal $0$, we obtain
	\begin{align*}
	\pi^* =& \frac{2 H(t,\lambda)(\mu_0 - r)  y}{   G(t) \sigma_0^2  }, \\
	u^* =& \frac{2 H(t,\lambda) \kappa_r \mu_1 \lambda  y}{ G(t)  \sigma_1^2\frac{ \rho \mu_2}{\delta}  }, \\
	v^* =&  \frac{1}{k}\bigg(\frac{2 H(t,\lambda) \iota_r \mu_2    y }{   G(t) \sigma_2^2 } + \frac{2 H_{\lambda} y }{   G(t) } + \frac{I_{\lambda}}{G(t)}\bigg),
	\end{align*}
	and the maximum of \eqref{mmvSN.diffusion.generator.suppose.2} is attained at $(\pi^*,u^*,v^*)$.
	
	\par Let us substitute $\pi^*$, $u^*$ and $v^*$ back into \eqref{mmvSN.diffusion.generator.suppose.2},
	\begin{align*}
	\mathcal{L}_t^{a,b} W(t,x,y,\lambda) =& G_t xy + H_t y^2 +I_t y +\biggl\{ r x  + (- \kappa_r + \kappa)\mu_1\lambda  + ( - \iota_r + \iota)k \mu_2\rho \biggr\} G(t) y \\
	&+ (-\delta \lambda + \rho \mu_2) ( H_{\lambda} y^2 +I_{\lambda} y ) + \frac{1}{2} \rho \sigma_2^2 ( H_{\lambda \lambda} y^2 +I_{\lambda \lambda} y )  \\
	&+ \frac{ H(t,\lambda)(\mu_0 - r)^2  y^2}{    \sigma_0^2  } + \frac{H(t,\lambda) \kappa_r^2 \mu_1^2 \lambda^2  y^2}{   \sigma_1^2\frac{ \rho \mu_2}{\delta}  } + \frac{H(t,\lambda) \rho \iota_r^2 \mu_2^2 y^2 }{   \sigma_2^2 } \\
	&+ 2 H_{\lambda} \iota_r \mu_2\rho  y^2 + I_{\lambda}\iota_r \mu_2\rho  y,
	\end{align*}
	then \eqref{mmvSN.diffusion.equation.function.value.coefficients} is obtained by separation of variables.
\end{proof}

\subsection{Proof of Lemma \ref{mmvSN.diffusion.lemma.coefficients}}

\begin{proof}
	It is easy to obtain that 
	\begin{align}
	\label{mmvSN.diffusion.equation.function.value.coefficients.suppose.G}
	G(t) = e^{r(T-t)}.
	\end{align}
	
	
	\par For $H(t,\lambda)$, we suppose that it has the following form
	\begin{align}
	\label{mmvSN.diffusion.equation.function.value.coefficients.suppose.H}
	H(t,\lambda) = \frac{1}{2 \theta}e^{\xi(t) \lambda^2 + \eta(t) \lambda + \zeta(t)}.
	\end{align}
	
	\par Substituting \eqref{mmvSN.diffusion.equation.function.value.coefficients.suppose.H} into \eqref{mmvSN.diffusion.equation.function.value.coefficients} gives
	%
	\begin{align*}
	\begin{cases}
	\xi'(t) + 2 \rho \sigma_2^2  \xi(t)^2 - 2 \delta \xi(t)  + \frac{\kappa_r^2  \mu_1^2  \delta}{\rho \mu_2 \sigma_1^2 } = 0, \\
	\eta'(t) -(\delta - 2 \rho \sigma_2^2 \xi(t)) \eta(t) + 2 \rho \mu_2 ( 2 \iota_r + 1) \xi(t) = 0, \\
	\zeta'(t) + \rho \mu_2 ( 2 \iota_r + 1) \eta(t) + \frac{1}{2} \rho \sigma_2^2 \eta(t)^2 + \rho \sigma_2^2  \xi(t) + \frac{ (\mu_0 - r)^2 }{    \sigma_0^2  }  + \frac{ \rho \iota_r^2 \mu_2^2 }{   \sigma_2^2 }= 0.
	\end{cases}
	\end{align*}
	
	\par For $\xi(t)$, we have
	\begin{align*}
	(T-t) =& - \int_t^T \frac{d\xi(s)}{  2 \rho \sigma_2^2\xi(s)^2 - 2 \delta \xi(s) + \frac{\kappa_r^2  \mu_1^2  \delta}{\rho \mu_2 \sigma_1^2}}. \\
	\end{align*}
	
	\par The characteristic equation is given by
	\begin{align*}
	2 \rho \sigma_2^2 d^2 - 2 \delta d + \frac{\kappa_r^2  \mu_1^2  \delta}{\rho \mu_2 \sigma_1^2} = 0,
	\end{align*}
	which has two solutions, namely,
	\begin{align*}
	d_{1,2} = \frac{2 \delta \pm \sqrt{\Delta}}{4 \rho \sigma_2^2},
	\end{align*}
	where
	\begin{align*}
	\Delta :=& 4 \delta^2 -   \frac{8 \kappa_r^2  \mu_1^2 \sigma_2^2 \delta}{\mu_2 \sigma_1^2}.
	\end{align*}
	
	If $\Delta = 0$, then $d_1 = d_2 = \frac{\delta}{2 \rho \sigma_2^2}$, and
	\begin{align*}
	(T-t) =& - \int_t^T \frac{d\xi(s)}{  2 \rho \sigma_2^2\xi(s)^2 - 2 \delta \xi(s) + \frac{\kappa_r^2  \mu_1^2  \delta}{\rho \mu_2 \sigma_1^2}} \\
	=& -\frac{1}{2 \rho \sigma_2^2}\int_t^T \frac{1}{(\xi(s)-d_1)^2} d\xi(s) \\
	=& -\frac{1}{2 \rho \sigma_2^2} \bigg(\frac{1}{d_1} + \frac{1}{\xi(t)-d_1}\bigg),
	\end{align*}
	thus
	\begin{align*}
	\xi(t) 
	=& \frac{\kappa_r^2 \mu_1^2}{\sigma_1^2} \frac{T-t}{\rho \mu_2(T-t) + \frac{\rho \mu_2}{\delta}}.
	\end{align*}

	If $\Delta > 0$,
	\begin{align*}
	T-t =& - \int_t^T \frac{d\xi(s)}{  2 \rho \sigma_2^2\xi(s)^2 - 2 \delta \xi(s) + \frac{\kappa_r^2  \mu_1^2  \delta}{\rho \mu_2 \sigma_1^2}} \\
	=& -\frac{1}{2 \rho \sigma_2^2(d_1-d_2)}\int_t^T \frac{1}{\xi(s)-d_1} - \frac{1}{\xi(s)-d_2}d\xi(s) \\
	=& -\frac{1}{\sqrt{\Delta}} \ln\left |\frac{\xi(s)-d_1}{\xi(s)-d_2}\right ||_t^T.
	\end{align*}
	
	\par Since $\xi(t)$ is a continuous function with boundary condition $\xi(T) = 0 < d_2 < d_1$, we assert that $\xi(t) = 0 < d_2 < d_1$ for all $t \in [0,T]$. Factually, if $\{t:\xi(t) \geq d_2\} \cap [0,T] \neq \varnothing$, then, by its continuity, there exists a $t_0 \in [0,T]$ such that $\xi(t_0) = d_2$. In this case,
	\begin{align*}
	\sqrt{\Delta}(T-t_0) = \ln \frac{d_1-\xi(t_0)}{d_2-\xi(t_0)} - \ln \frac{d_1}{d_2} = -\infty,
	\end{align*}
	which is a contradiction. Hence
	\begin{align*}
	\sqrt{\Delta}(T-t) =&  \ln \frac{d_1-\xi(t)}{d_2-\xi(t)} - \ln \frac{d_1}{d_2},
	\end{align*}
	which gives
	\begin{align*}
	\xi(t) = \frac{d_1d_2(e^{\sqrt{\Delta}(T-t)}-1)}{d_1e^{\sqrt{\Delta}(T-t)}-d_2}.
	\end{align*}
	\par Moreover since $d_1 > d_2$, we have $\xi(t) >0$ for all $t \in [0,T]$.
	
	\par For $\eta(t)$, by integration, we have
	\begin{align*}
	\eta(t) =& \int_t^T e^{-\int_t^s \left (\delta - 2 \rho \sigma_2^2 \xi(\tau)\right ) d\tau} 2 \rho \mu_2 ( 2 \iota_r + 1) \xi(s) ds.
	\end{align*}
	\par If $\Delta = 0$, we have
	\begin{align*}
	\delta =  \frac{2 \kappa_r^2  \mu_1^2 \sigma_2^2}{\mu_2 \sigma_1^2}.
	\end{align*}
	Hence
	\begin{align*}
	2 \rho \sigma_2^2 \int_t^s \xi(\tau) d\tau =&  \frac{2 \rho \sigma_2^2 \kappa_r^2 \mu_1^2}{\sigma_1^2} \int_t^s \frac{T-\tau}{\rho \mu_2(T-\tau) + \frac{\rho \mu_2}{\delta}} d\tau \\
	=& \frac{2 \sigma_2^2 \kappa_r^2 \mu_1^2}{\sigma_1^2 \mu_2} \left(s-t + \frac{1}{\delta} \ln\left(\frac{\delta(T-s)+1}{\delta(T-t)+1}\right)\right) \\
	=& \delta (s-t) + \ln\left(\frac{\delta(T-s)+1}{\delta(T-t)+1}\right),
	\end{align*}
	and
	\begin{align*}
	e^{-\int_t^s \left (\delta - 2 \rho \sigma_2^2 \xi(\tau)\right ) d\tau} 
	=& \frac{\delta(T-s)+1}{\delta(T-t)+1}.
	\end{align*}
	Therefore,
	\begin{align*}
	\eta(t) =& \int_t^T e^{-\int_t^s \left (\delta - 2 \rho \sigma_2^2 \xi(\tau)\right ) d\tau} 2 \rho \mu_2 ( 2 \iota_r + 1) \xi(s) ds \\
	=& 2 \rho \mu_2 ( 2 \iota_r + 1) \int_t^T \frac{\delta(T-s)+1}{\delta(T-t)+1} \frac{\kappa_r^2 \mu_1^2}{\sigma_1^2} \frac{T-s}{\rho \mu_2(T-s) + \frac{\rho \mu_2}{\delta}} ds \\
	=& \frac{2  ( 2 \iota_r + 1)\kappa_r^2 \mu_1^2}{\sigma_1^2}  \int_t^T \frac{\delta(T-s)+1}{\delta(T-t)+1}  \frac{\delta(T-s)}{\delta(T-s) + 1} ds \\
	=& \frac{ ( 2 \iota_r + 1)\delta \kappa_r^2 \mu_1^2}{\sigma_1^2} \frac{(T-t)^2 }{\delta(T-t)+1}.
	\end{align*}
	If $\Delta > 0$, then
	\begin{align*}
	2 \rho \sigma_2^2 \int_t^s \xi(\tau) d\tau =& 2 \rho \sigma_2^2 \int_t^s \left (d_1 + (d_2-d_1)\frac{e^{\sqrt{\Delta}(T-r)}}{e^{\sqrt{\Delta}(T-\tau)}-\frac{d_2}{d_1}} \right ) d\tau \\
	=& 2 \rho \sigma_2^2 \int_t^s d_1  d\tau - 2 \rho \sigma_2^2 \int_t^s \left ( (d_1-d_2)\frac{e^{\sqrt{\Delta}(T-r)}}{e^{\sqrt{\Delta}(T-\tau)}-\frac{d_2}{d_1}} \right ) d\tau \\
	=& 2 \rho \sigma_2^2 d_1(s-t) +  \ln{\left (\frac{d_1e^{\sqrt{\Delta}(T-s)}-d_2}{d_1e^{\sqrt{\Delta}(T-t)}-d_2}\right )},
	\end{align*}
	and
	\begin{align*}
	e^{-\int_t^s \left (\delta - 2 \rho \sigma_2^2 \xi(\tau)\right ) d\tau} 
	=& e^{\frac{\sqrt{\Delta}}{2}(t-s)} \left (\frac{d_1e^{\sqrt{\Delta}(T-s)}-d_2}{d_1e^{\sqrt{\Delta}(T-t)}-d_2}\right ).
	\end{align*}
	Thus $\eta(t)$ can be simplified as
	\begin{align*}
	\eta(t) =& \int_t^T e^{-\int_t^s \left (\delta - 2 \rho \sigma_2^2 \xi(\tau)\right ) d\tau} 2 \rho \mu_2 ( 2 \iota_r + 1) \xi(s) ds \\
	=& 2 \rho \mu_2 ( 2 \iota_r + 1) \int_t^T e^{\frac{\sqrt{\Delta}}{2}(t-s)}  \left (\frac{d_1e^{\sqrt{\Delta}(T-s)}-d_2}{d_1e^{\sqrt{\Delta}(T-t)}-d_2}\right ) \left (\frac{d_1d_2(e^{\sqrt{\Delta}(T-s)}-1)}{d_1e^{\sqrt{\Delta}(T-s)}-d_2} \right ) ds \\
	=& \frac{4 \rho \mu_2 ( 2 \iota_r + 1) d_1d_2}{3 \sqrt{\Delta}}\frac{  \left ( e^{\frac{\sqrt{\Delta}}{2}(T-t)} -1\right )^2 \left ( 1 +2 e^{-\frac{\sqrt{\Delta}}{2}(T-t)} \right )}{d_1e^{\sqrt{\Delta}(T-t)}-d_2}.
	\end{align*}
	
	\par By integrating with respect $t$, $\zeta(t)$ can be obtained as follows
	\begin{align*}
	\zeta(t) = \rho \int_t^T \bigg(\rho \mu_2 ( 2 \iota_r + 1) \eta(s) + \frac{1}{2} \rho \sigma_2^2 \eta(s)^2 + \rho \sigma_2^2  \xi(s) \bigg) ds + \frac{ \rho (\mu_0 - r)^2 }{    \sigma_0^2  }(T-t)  + \frac{ \rho^2 \iota_r^2 \mu_2^2 }{   \sigma_2^2 } (T-t).
	\end{align*}

	\par To sum up, we have
	\begin{align*}
	\begin{cases}
	\xi(t) =& \frac{\kappa_r^2 \mu_1^2}{\sigma_1^2} \frac{T-t}{\rho \mu_2(T-t) + \frac{\rho \mu_2}{\delta}} 1_{\{\Delta = 0\}} + \frac{d_1d_2(e^{\sqrt{\Delta}(T-t)}-1)}{d_1e^{\sqrt{\Delta}(T-t)}-d_2} 1_{\{\Delta >0\}}, \\
	\eta(t) =& \frac{ ( 2 \iota_r + 1)\delta \kappa_r^2 \mu_1^2}{\sigma_1^2} \frac{(T-t)^2 }{\delta(T-t)+1} 1_{\{\Delta = 0\}} + \frac{4 \rho \mu_2 ( 2 \iota_r + 1) d_1d_2}{3 \sqrt{\Delta}}\frac{  \left ( e^{\frac{\sqrt{\Delta}}{2}(T-t)} -1\right )^2 \left ( 1 +2 e^{-\frac{\sqrt{\Delta}}{2}(T-t)} \right )}{d_1e^{\sqrt{\Delta}(T-t)}-d_2} 1_{\{\Delta >0\}}, \\
	\zeta(t) =& \rho \int_t^T \bigg(\rho \mu_2 ( 2 \iota_r + 1) \eta(s) + \frac{1}{2} \rho \sigma_2^2 \eta(s)^2 + \rho \sigma_2^2  \xi(s) \bigg) ds + \frac{ \rho (\mu_0 - r)^2 }{    \sigma_0^2  }(T-t)  + \frac{ \rho^2 \iota_r^2 \mu_2^2 }{   \sigma_2^2 } (T-t),
	\end{cases}
	\end{align*}
	where
	\begin{align*}
	\Delta = 4 \delta^2 -   \frac{8 \kappa_r^2  \mu_1^2 \sigma_2^2 \delta}{\mu_2 \sigma_1^2}, \quad d_{1,2} = \frac{2 \delta \pm \sqrt{\Delta}}{4 \rho \sigma_2^2}.
	\end{align*}
	
	\par We make the ansatz that $I(t,\lambda)$ has the following form
	\begin{align}
	\label{mmvSN.diffusion.equation.function.value.coefficients.suppose.I}
	I(t,\lambda) = \alpha(t) \lambda +\beta(t).
	\end{align}
	
	\par Similarly, substituting \eqref{mmvSN.diffusion.equation.function.value.coefficients.suppose.G} and \eqref{mmvSN.diffusion.equation.function.value.coefficients.suppose.I} into \eqref{mmvSN.diffusion.equation.function.value.coefficients} gives
	%
	\begin{align*}
	\begin{cases}
	\alpha'(t) - \delta \alpha(t) - (\kappa_r - \kappa)\mu_1 e^{r(T-t)}  = 0, \\
	\beta'(t) - (\iota_r - \iota) k \rho \mu_2 e^{r(T-t)} + ( 1 + \iota_r ) \rho \mu_2 \alpha(t)  = 0,
	\end{cases}
	\end{align*}
	which admit the solutions
	\begin{align*}
	\begin{cases}
	\alpha(t) = - (\kappa_r - \kappa)\mu_1 \int_t^T e^{-\delta(s-t)} e^{r(T-s)} ds , \\
	\beta(t) = \rho \int_t^T  \biggl\{\alpha(s)(1 + \iota_r) \mu_2 - ( \iota_r - \iota)k \mu_2 e^{r(T-s)}\biggr\} ds,
	\end{cases}
	\end{align*}
	i.e.,
	\begin{align*}
	\begin{cases}
	\alpha(t) = \frac{(\kappa_r - \kappa)\mu_1}{\delta + r} (e^{- \delta(T-t)} - e^{r(T-t)})  , \\
	\beta(t) = \frac{\rho (\kappa_r - \kappa) (\iota_r+1) \mu_1 \mu_2}{\delta + r} \bigg( \frac{1}{\delta}(1-e^{-\delta(T-t)})+\frac{1}{r}(1-e^{r(T-t)}) \bigg)  - \frac{\rho ( \iota_r - \iota)k\mu_2}{r} (1-e^{r(T-t)}).
	\end{cases}
	\end{align*}

\end{proof}


\end{appendices}

\end{document}